\documentclass[11pt,reqno]{amsart}
\usepackage{amsmath,amsfonts,euscript,amscd,amsthm,amssymb,upref}

\newtheorem{theorem}{Theorem}[section]
\newtheorem{lemma}[theorem]{Lemma}
\newtheorem{proposition}[theorem]{Proposition}
\newtheorem{corollary}[theorem]{Corollary}
\theoremstyle{definition}
\newtheorem{definition}[theorem]{Definition}
\theoremstyle{example}
\newtheorem{example}[theorem]{Example}
\theoremstyle{remark}
\newtheorem{remark}[theorem]{Remark}
\theoremstyle{formula}
\newtheorem{formula}[theorem]{}
\newtheorem{assum}[theorem]{Assumption}
\newtheorem{review}[theorem]{}

\numberwithin{equation}{section}

\newcommand{\ms}{\medskip} 
\newcommand{\sm}{\smallskip}

\newcommand{\si}{^\sigma}
\newcommand{\msi}{^{-\sigma}}
\newcommand{\M}{\mathcal{M}}
\newcommand{\GM}{{\rm gr}\mathcal{M}}
\newcommand{\R}{{\it Rad \,}}
\newcommand{\GR}{{\rm gr}{\it Rad \,}}
\newcommand{\T}{\mathcal{T}}

\newcommand{\al}{\alpha}
\newcommand{\be}{\beta}
\newcommand{\Ga}{\Gamma}
\newcommand{\ga}{\gamma}
\newcommand{\de}{\delta}
\newcommand{\De}{\Delta}
\newcommand{\e}{\epsilon}
\newcommand{\ep}{\varepsilon}
\newcommand{\la}{\lambda}
\newcommand{\La}{\Lambda}

\newcommand{\op}{\oplus}
\newcommand{\wh}{\widehat}
\newcommand{\hJ}{{\wh J}}
\newcommand{\hj}{\hJ}
\newcommand{\ts}{\textstyle}
\newcommand{\wt}{\widetilde}
\newcommand{\Id}{{\rm Id}}
\newcommand{\Mat}{{\rm Mat}}
\newcommand{\ch}{\sp{\scriptscriptstyle\vee}}
\newcommand{\alg}{_{\rm alg}}
\newcommand{\rmb}{{\rm B}}
\newcommand{\lan}{\langle}
\newcommand{\ran}{\rangle}

\newcommand{\scL}{\mathcal{L}}
\newcommand{\scP}{\mathcal{P}}
\newcommand{\scQ}{\mathcal{Q}}
\newcommand{\scS}{\mathcal{S}}
\newcommand{\scT}{\mathcal{T}}

\DeclareMathOperator{\TKK}{TKK} 
\newcommand{\frK}{\TKK}
\DeclareMathOperator{\supp}{supp} 
\DeclareMathOperator{\rmbc}{BC} 
\DeclareMathOperator{\End}{End} \DeclareMathOperator{\AC}{AC}
\DeclareMathOperator{\FC}{FC} \DeclareMathOperator{\rmH}{H}

\newcommand{\frS}{\mathfrak{S}} 
\newcommand{\fru}{\mathfrak{p}}
\newcommand{\frsu}{\mathfrak{sp}}

\newcommand{\eso}{\mathfrak{eo}}
\newcommand{\fro}{\mathfrak{o}}

\newcommand{\ZZ}{\mathbb{Z}}
\newcommand{\CC}{\mathbb{C}}

\begin{document}

\title[Lie tori of type $\rmb_2$ and
graded-simple Jordan structures]{Lie tori of type $\rmb_2$ and
graded-simple Jordan structures covered by a triangle}

\author{ Erhard Neher\and Maribel Toc\'on}
\address{Department of Mathematics and Statistics,
University of Ottawa, Ottawa, Ontario K1N 6N5, Canada}
\email{neher@\allowbreak uottawa.ca}
\thanks{The first author is partially supported by a Discovery Grant of the
Natural Sciences and Engineering Research Council of Canada, and the
second  by  MEC and FEDER (MTM2007-61978)}

\address{Departamento de Estad\'istica e Investigaci\'on Operativa,
Universidad de C\'ordoba, Puerta Nueva s/n, C\'ordoba, 14071,
Spain.} \email{td1tobam@uco.es}

{\renewcommand{\thefootnote}{} }

{\renewcommand{\thefootnote}{} }

\date{\today}

\maketitle
\newcommand{\blankbox}[2]{%
  \parbox{\columnwidth}{\centering
   \setlength{\fboxsep}{0pt}%
    \fbox{\raisebox{0pt}[#2]{\hspace{#1}}}%
}%
}

\begin{abstract}  We classify two classes of $\rmb_2$-graded Lie algebras which have a
second compatible grading by an abelian group $\La$: (a)
$\La$-graded-simple, $\La$ torsion-free and (b)
division-$\La$-graded. Our results describe the centreless cores of
a class of affine reflection Lie algebras, hence apply in particular
to the centreless cores of extended affine Lie algebras, the
so-called Lie tori, for which we recover results of Allison-Gao and
Faulkner. Our classification (b) extends a recent result of
Benkart-Yoshii.

Both classifications are consequences of a new description of Jordan
algebras covered by a triangle, which correspond to these Lie
algebras via the Tits-Kantor-Koecher construction. The Jordan
algebra classifications follow from our results on
graded-triangulated Jordan triple systems. They generalize work of
McCrimmon and the first author as well as the
Osborn-McCrimmon-Capacity-$2$-Theorem in the ungraded case.
\end{abstract}

\section*{Introduction}

This paper deals with two related algebraic objects, Lie algebras
graded by the root system $\rmb_2$ and Jordan structures (algebras,
triple systems and pairs), covered by a triangle of idempotents,
respectively tripotents. Its aim is to classify the graded-simple
structures in these categories. \sm

On the Lie algebra side, the motivation for this paper comes from
the theory of extended affine Lie algebras, which generalize affine
Kac-Moody Lie algebras and toroidal Lie algebras (\cite{aabgp}), and
the even more general affine reflection Lie algebras (\cite{Npers},
LA=Lie algebra):
$$
 \text{affine LA} \subset \text{extended affine LA} \subset
   \text{affine reflection LA}
$$
As explained in \cite[\S6]{Npers}, an important ingredient in the
structure theory of these Lie algebras is a certain subquotient, the
so-called centreless core, whose structure is, respectively, as
follows:
\begin{eqnarray*}
  && \text{un/twisted loop algebra} \subset \text{centreless Lie torus}
 \\ && \quad \subset
  \text{centreless predivision-root-graded Lie algebra}
\end{eqnarray*}
It is therefore of interest to understand the structure of Lie tori
or predivision-root-graded Lie algebras. As the reader will see
below, these Lie algebras are defined in terms of a finite
irreducible root system. One of the goals of this paper is to
elucidate the structure of root-graded Lie algebras of type
$\rmb_2$, which is one of the most complicated cases.

Let $\De$ be the root system $\rmb_2$ and put $R=\De \cup \{0\}$ (we
assume here $0\notin \De$). Also, let $\La$ be an abelian group. We
consider Lie algebras defined over a ring $k$ with $\frac{1}{2}$ and
$\frac{1}{3}\in k$ which have a decomposition
\begin{enumerate}
 \item[(RG1)] $L=\bigoplus_{\al \in R, \, \la \in \La} L_\al^\la$
with $[L_\al^\la, \, L_\be^\mu] \subset L_{\al + \be}^{\la + \mu}$
($= 0$ if $\al + \be \not\in R$), satisfying

\item[(RG2)] $L_0 = \sum_{\al \in \De} [L_\al, L_{-\al}]$.
 \end{enumerate} We call $0\ne e\in L_\al^\la$, $\al \in \De$, {\it invertible\/} if there exists
$f\in L_{-\al}^{-\la}$ such that $h=[e,\, f]$ acts on $x_\be \in
L_\be^\mu$, $\be\in R$, as $[h,x_\be] = \lan \be,\al\ch\ran x_\be$
for $\lan \be,\al\ch\ran$ the Cartan integer of $\al,\be\in R$. This
perhaps unusual definition is justified by examples, in which
invertible elements as defined above are given in terms of
invertible elements of coordinate algebras. A Lie algebra satisfying
(RG1) and (RG2) is called
\begin{enumerate}

\item[-] {\it $\rmb_2$-graded with a compatible $\La$-grading\/} if every
$L_\al^0$, $\al\in \De$, contains an invertible element,

\item[-] {\it $\rmb_2$-graded-simple\/} if $L$ is $R$-graded with a
compatible $\La$-grading and if  $\{0\}$ and $L$ are the only
$\La$-graded ideals of $L$,

\item[-] {\it predivision-$\rmb_2$-graded\/} if $L$ is $R$-graded with a
compatible $\La$-grading and every $0\ne L_\al^\la$, $\al \in \De$,
contains an invertible element,

\item[-] {\it division-$\rmb_2$-graded\/} if  $L$ is $R$-graded with a
compatible $\La$-grading  and every nonzero element in $L_\al^\la$,
$\al \in \De$, is invertible,

\item[-] a {\it Lie torus of type $\rmb_2$\/} if $k$ is a field, $L$ is
division-$\rmb_2$-graded and $\dim_k L_\al^\la \le 1$ for all $\al
\in \De$.
\end{enumerate}

If $k$ is a field of characteristic $0$, a split simple Lie algebra
of type $\rmb_2$ is clearly $\rmb_2$-graded in the sense above. So
we have two isomorphic examples, the orthogonal Lie algebra
$\fro_5(k)$ and the symplectic Lie algebra $\frsu_2(k)$, using
$\rmb_2= {\rm C}_2$. Another example is not far off: We can re-grade
the standard root space decomposition of $\mathfrak{sl}_4(k)$,
viewed as an ${\rm A}_3$-grading, by ``folding'' ${\rm A}_3$ into
$\rmb_2$, thus giving $\mathfrak{sl}_4(k)$ the structure of a simple
$\rmb_2$-graded Lie algebra.

The reader will not be very surprised in learning that one can
replace $k$ in the three examples $\mathfrak{sl}_4(k)$, $\frsu_2(k)$
and $\fro_5(k)$ by  more general, not necessarily commutative
coordinates and still gets a simple $\rmb_2$-graded Lie algebra. In
doing so, $\frsu_2(.)$ and $\fro_5(.)$ will no longer be isomorphic.
Our first classification result says that, up to central extensions
and allowing graded-simple coordinates, these are all examples in
the graded-simple case with $\La$ torsion-free: \sm

{\bf Theorem A} (Th.~\ref{secfinres}) {\it Let $\La$ be torsion-free
and let $L$ be a Lie algebra over a ring $k$ containing
$\frac{1}{2}$ and $\frac{1}{3}$. Then $L$ is
 centerless $\rmb_2$-graded-simple if and only if $L$ is graded isomorphic to

\begin{enumerate}

\item[\rm (I)] $\mathfrak{sl}_4(A)/Z\big( \mathfrak{sl}_4(A)\big)$
where $\mathfrak{sl}_4(A) = \{ X\in \mathfrak{gl}_4(A): {\rm tr}(X)
\in [A,A]\}$, and $A$ is a graded-simple associative $k$-algebra,

 \item[\rm (II)] $\frsu_2(A,\pi)/Z(\frsu_2(A,\pi))$ for
$A$ as in {\rm (I)} with involution $\pi$,

\item[\rm (III)] an elementary orthogonal Lie algebra $\eso(Q)$
where $Q$ is a graded-nondegenerate quadratic form with base point
and containing two hyperbolic planes over a graded-field.
\end{enumerate} \sm}

We point out that even the case $\La=\{0\}$ was not explicitly known
before, although it could have been derived from \cite{MNSimple}.
Moreover, in the application to affine reflection Lie algebras and
their centreless core our assumption on $\La$ is fulfilled. \sm

Our second classification result (Th.~\ref{findivres}) is parallel
to Th.~A: It allows an arbitrary $\La$, but assumes that the Lie
algebra $L$ is centreless and division-$\rmb_2$-graded. In this
setting, case (I) disappears, the algebra $A$ in (II) is
division-graded and $Q$ in (III) is graded-anisotropic. In
characteristic $0$ this result has also been obtained by
Benkart-Yoshii \cite[Th.~4.3]{BY} using different methods and giving
a less precise description of the Lie algebras. We can easily derive
from our results a classification of centreless Lie tori of type
$\rmb_2$. \sm

{\bf Corollary} (Cor.~\ref{finalres}) {\it A Lie algebra $L$ is a
centreless Lie torus of type $\rmb_2$ if and only if $L$ is graded
isomorphic to one of the following:
\begin{enumerate}
 \item[\rm (I)] a symplectic Lie algebra $\frsu_2(A,\pi)$ for
$A$ a noncommutative associative torus with involution $\pi$, or to
\item[\rm (II)] an elementary orthogonal Lie algebra
$\eso(Q)$ for $Q$ a graded-aniso\-tro\-pic quadratic form  over an
associative torus with the same properties as $Q$ in {\rm (III)} of
{\rm Th.~{\rm A}}.
\end{enumerate}}

Again in characteristic $0$ this has also been obtained by
Benkart-Yoshii \cite[Th.~5.9]{BY}. For $k=\CC$ and $\La=\ZZ^n$, the
Lie tori classification is due to Allison-Gao \cite{AG}. A different
approach to this classification has recently been given by Faulkner
\cite{F} in the context of his classification of Lie tori of type
${\rm BC}_2$. \ms

We obtain the Lie algebra results as a consequence of our results on
so-called triangulated Jordan structures, by linking Lie algebras to
Jordan structures via the Tits-Kantor-Koecher construction. This
brings us to the second goal for this paper, the classification of
graded-simple-triangulated Jordan structures. In this introduction,
we restrict ourselves to Jordan algebras for which the results are
easier to state. A quadratic unital Jordan algebra $J$ is called
{\it graded-triangulated\/} if $J=\bigoplus_{\la \in \La} J^\la$ is
graded by some abelian group $\La$ and contains two supplementary
orthogonal idempotents $e_1, e_2\in J^0$ strongly connected by some
$u\in J_1(e_1) \cap J_1(e_2) \cap J^0$. Of course, {\it
graded-simple-triangulated\/} means graded-simple and
graded-triangulated. \sm

Analogously to the Lie algebra case, two isomorphic examples of
triangulated Jordan algebras are immediate for $k$ algebraically
closed, the Jordan algebra $\rmH_2(k)$ of $2\times 2$ symmetric
matrices and the Jordan algebra $\AC\alg(k)$ of a $3$-dimensional
nondegenerate quadratic form. In addition, also the Jordan algebra
$\Mat_2(k)$ is naturally triangulated. And as in the Lie algebra
case, the reader will not be surprised to learn that one can replace
$k$ by more general coordinates and still gets a graded-simple
triangulated Jordan algebra.
 \sm

{\bf Theorem B} (Th.~\ref{classalgth}.b) {\it A triangulated
quadratic Jordan algebra $J$ which is graded-simple with respect to
a grading by a torsion-free abelian group $\La$ is graded isomorphic
to one of the following Jordan algebras:
\begin{enumerate}

\item[{\rm (I)}] full matrix algebra $\Mat_2(A)$ for a
 noncommutative graded-simple associative  unital $A$;

\item[{\rm (II)}] hermitian matrices $\rmH_2(A,A_0, \pi )$ for a  noncommutati\-ve
graded-simple associative unital $A$ with ample subspace $A_0$ and
graded involution $\pi$;

\item[{\rm (III)}] Clifford Jordan algebra $\AC\alg(q,F, F_0)$
for a graded-nondegenerate $q$ over a graded-field $F$ with
Clifford-ample subspace $F_0$.
\end{enumerate}
\noindent Conversely, all Jordan algebras in {\rm (I)--(III)} are
graded-simple-triangu\-la\-ted. } \sm

To put Theorem B into perspective, let us point out that already the
case $\La=\{0\}$ is nontrivial: The
``Osborn-McCrimmon-Capacity-$2$-Theorem'' (\cite[6.3]{ark},
\cite[22.2]{taste}), which classifies simple Jordan algebras of
capacity $2$ (= simple triangulated Jordan algebras with division
diagonal Peirce spaces) is the most complicated piece of the
classification of simple Jordan algebras with capacity, and a
cornerstone of the classification of simple Jordan algebras.

As for Th.~A, we prove a second classification result
(Cor.~\ref{divalg}) in which the assumption on $\La$ is replaced by
the condition that elements in the Peirce space $J_{12}$ are sums of
invertible elements. And of course, there are also corollaries for
triangulated Jordan algebra tori, formulated in Cor.~\ref{jazncla}
for $\La=\ZZ^n$. \sm

We have mentioned that the Lie algebra results follow from our
results on Jordan structures, viz., the cases (I)--(III) in Theorem
A correspond to the cases (I)--(III) in Theorem B. So how do we
prove the Jordan algebra result? In fact, we first prove the results
for graded-simple-triangulated Jordan triple systems by adapting the
approach of \cite{MNSimple} to the graded-simple setting. This paper
deals with the case $\La=\{0\}$ and so generalizes the
Osborn-McCrimmon-Capacity-$2$-Theorem to Jordan triple systems. Once
the triple system case has been established, we can derive the
results for Jordan algebras and Jordan pairs (Th~\ref{classpairs},
Cor.~\ref{divpair} and Cor.~\ref{jordantori}) by standard
techniques. \ms

The paper is divided into seven sections. In the first two sections
we establish the terminology,  identities and general results about
graded Jordan triple systems and graded-triangulated Jordan triple
systems, respectively. Proofs in the first two sections are mainly
left to the reader since they are easy generalizations of the
corresponding ungraded cases. In sections \S\ref{sec:trhms} and
\S\ref{sec:triclis} we present our two basic models for
graded-triangulated Jordan triple systems, the hermitian matrix
systems and the Clifford systems and prove Coordinatization Theorems
for both of them (Th.~\ref{hermitian} and Th.~\ref{clifford}).
Section \S\ref{sec:grsimcla} is devoted to classifying
graded-simple-triangulated Jordan triple systems. As a corollary we
obtain a classification of division-triangulated Jordan triple
systems (Cor.~\ref{divgradclassi}) and triangulated Jordan triple
tori (Cor.~\ref{jtstorclassi}). These classification theorems are
extended to Jordan algebras and Jordan pairs  in
\S\ref{sec:triangjp}. Finally, in the last section, we apply our
results to Lie algebras. \ms

Unless specified otherwise, all algebraic structures are defined
over an arbitrary ring of scalars, denoted $k$, and are assumed to
be graded by an abelian group $\Lambda$, written additively. We will
use Loos' Lecture Notes \cite{JP} as our basic reference for Jordan
triple systems and Jordan pairs.

\section{Graded Jordan triple systems} \label{sec:grJTS}

This section introduces some basic notions of graded Jordan triple
systems. For example we establish in Th.~\ref{grsiminh} that
Peirce-$2$- and Peirce-$0$-spaces of a degree $0$ tripotent inherit
graded-simplicity. \sm

A $k$-module $M$ is  {\it graded by $\La$} if $M= \bigoplus_{\la \in
\La}\, M^\la$ where $(M^\la : \la \in \La)$ is a family of
$k$-submodules of $M$. In this case, we call $M$ {\it
$\La$-graded\/} if the {\it support set} $\supp_\La \{ \la \in \La :
M^\la \ne 0 \}$ generates $\La$ as an abelian group. Of course, if
$M$ is graded by $\La$, it is $\Xi$-graded for $\Xi$ the subgroup
generated by $\supp_\La M$. But it is usually more convenient to
just consider graded modules (and triple systems) as opposed to
$\La$-graded ones. We say that {\it $M$ is graded\/} if $M$ is
graded by some (unimportant) abelian group, which for simplicity we
assume to be $\La$. A {\it homogeneous element} of a graded $M$ is
an element of $\bigcup_{\la \in \La} M^\la$. If
$M=\bigoplus_{\lambda \in \Lambda}M^\lambda$ and
$N=\bigoplus_{\lambda \in \Lambda}N^\lambda$ are graded modules, a
$k$-linear map $\varphi : M\rightarrow N$ is said to be {\it
homogeneous of degree $\ga \in \La$} if $\varphi (M^\la)\subseteq
N^{\la +\ga}$ for all $\la \in \La$.

A Jordan triple system $J$ with quadratic operator $P$ and triple
product $\{.,.,.\}$ is {\it graded by $\La$} if the underlying
module is so, say $J=\bigoplus_{\lambda \in \Lambda}J^\lambda$, and
the family $(J^\la : \la \in \La)$ satisfies $P(J^\lambda)J^\mu
\subseteq J^{2\lambda + \mu}$ and $\{J^\la , J^\mu , J^\nu
\}\subseteq J^{\la + \mu +\nu}$ for all $\lambda, \mu , \nu \in
\Lambda$. We will say that $J$ is {\it $\La$-graded\/} if $J$ is
graded by $\La$ and the underlying module is $\La$-graded. As for
modules, we will simply speak of a {\it graded Jordan triple
system\/} if the grading group $\La$ is not important.

If $J$ and $J'$ are graded Jordan triple systems, a homomorphism
$\varphi : J\rightarrow J'$ is said to be {\it graded} if it is
homogeneous of degree $0$. Correspondingly, a {\it graded
isomorphism}  is  a bijective graded homomorphism, and we say that
$J$ and $J'$ are {\it graded isomorphic}, written as $J\cong_\La
J'$,  if there exists a graded isomorphism between $J$ and $J'$.

Let  $J$ be a graded Jordan triple system. A subsystem $M$ of $J$ is
called {\it graded} if $M = \bigoplus_{\la \in \La} (M \cap J^\la)$.
If $M$ is an arbitrary subsystem  of $J$, the greatest graded
subsystem of $J$ contained in $M$ is $M^{\rm gr} = \bigoplus_{\la
\in \La} (M \cap J^\la)$. If $M$ is an ideal of $J$, then so is
$M^{\rm gr}$. We also note that the quotient of $J$ by a graded
ideal is again graded with  respect to the canonical quotient
grading. We call $J$ {\it graded-simple} if $P(J)J \ne 0$ and every
graded ideal is either $0$ or equal to $J$. We say that $J$ is {\it
graded-prime} if $P(I)K=0$ for graded ideals $I,K$ of $J$ implies
$I=0$ or $K=0$, and {\it graded-semiprime} if $P(I)I=0$ for a graded
ideal $I$ implies $I=0$. We denote by $\T (J)=\{x\in J: P(x)J=0\}$
the set of {\it trivial elements} of $J$, and put $\T^\Lambda
(J):=\bigcup_{\lambda \in \Lambda} \T^\lambda (J)$, where
$\T^\lambda (J)=\T (J)\cap J^\lambda$. We say that $J$ is {\it
graded-nondegenerate} if $\T^\Lambda (J)= 0$. We note that {\it if
$J$ is graded-nondegenerate, it is also graded-semiprime.} Finally,
we say that $J$ is {\it graded-strongly prime} if it is graded-prime
and graded-nondegenerate, and {\it division-graded} if it is nonzero
and every nonzero homogeneous element is invertible in $J$.

Recall that the {\it McCrimmon radical}  $\M (J)$ of a Jordan triple
system $J$ is the smallest ideal of $J$ such that the quotient $J/\M
(J)$ is nondegenerate \cite[\S4]{JP}. It can be constructed as
follows: $\M (J):=\bigcup_\alpha\M_\alpha(J)$, where $\M_0(J)=0$,
$\M_1(J)$ is the ideal of $J$ generated by the set of trivial
elements $\T (J)$ of $J$ and, by using transfinite induction, the
ideals $\M_\alpha (J)$ are defined by $\M_\alpha(J)/\M_{\alpha -
1}(J)=\M_1\big(J/\M_{\alpha-1}(J)\big)$ if $\al$ is a non-limit
ordinal and $\M_\alpha (J)=\bigcup _{\beta < \alpha}\M_\beta(J)$ for
a limit ordinal $\alpha$.

\begin{definition} Let $J$ be a  graded  Jordan triple system
and let $\M (J)$ be its  McCrimmon radical. We define the {\it
graded-McCrimmon radical} of $J$, denoted ${\GM} (J)$, as the
greatest graded ideal contained in $\M(J)$, i.e., $${\GM} (J):=
\M(J)^{\rm gr} = \textstyle \bigoplus_{\la \in \La} \big(J^\la \cap
\M(J)\big).$$ Thus $\GM(J) = \bigoplus_{\la \in \La}\, \GM^\la(J)$
for $\GM^\la(J) = J^\la \cap \M(J)$. The following characterization
is immediate {f}rom the definition, see \cite[\S4]{JP} for the
ungraded case.
\end{definition}

\begin{proposition} Let $J$ be a  graded  Jordan triple system.
Then the homogenous spaces $\GM^\la(J)$ of\/ ${\GM}(J)$ are
$\GM^\la(J) = \bigcup_\al (\M_\al(J) \cap J^\la)$  for $\M_\al(J)$
as defined above.   The graded-McCrimmon radical is the smallest
graded ideal of $J$ such that the quotient $J/{\GM}(J)$ is
graded-nondegenerate.
\end{proposition}

We will need the following result.

\begin{proposition}\label{nonde}  {\rm (see
\cite{localnil}, \cite{KZ} for $\La=0$)}  If $J$ is  a graded-simple
Jordan triple system,  then $J$ is graded-nondegenerate.
\end{proposition}

\begin{proof}  We can suppose ${\GM}(J)=J$.
Hence ${\GM}(J)=\M (J)=J$. By \cite{localnil}, \cite{KZ} $J$ is then
{\it locally nilpotent}, i.e., every finitely generated subalgebra
of $J$ is nilpotent. But this immediately leads to a contradiction.
\end{proof}

Let $e$ be a tripotent in a Jordan triple system $J$, i.e.,
$P(e)e=e$. We thus have the Peirce decomposition of $J$ with respect
to $e$, written as $J=J_2(e)\oplus J_1(e)\oplus J_0(e)$.  If, in
addition, $J$ is  graded  and $e\in J^0$, it is immediate that that
the Peirce  spaces $J_i(e)$, $i=0,1,2$, are  graded:
 $J_i(e)=\bigoplus_{\lambda \in \Lambda}J_i^{\lambda}(e)$ where
$J_i^{\lambda}(e)=J_i(e)\cap J^\lambda$.

\begin{theorem}\label{grsiminh} Let $J$ be a graded-simple Jordan triple system
with a tripotent  $0\neq e\in J^0$. Then the Peirce subsystem
$J_2(e)$ is graded-simple and if $J_0(e)\neq 0$, then $J_0(e)$ is
also graded-simple.\end{theorem}

\begin{proof} This can be proven in the same way as the ungraded
result \cite[3.8]{MSimple}. \end{proof}

\section{Graded-triangulated Jordan triple systems} \label{sec:gr-triang
JTS general}

In this section we begin our study of graded-triangulated Jordan
triple systems. We define the basic notations used throughout,
present a list of multiplication rules, and discuss
graded-nondegeneracy and graded-simplicity (Prop.~\ref{peirce}).
Throughout $J$ is a Jordan triple system, assumed to be graded from
Def.~\ref{defgratri} on. \sm

A triple of  nonzero tripotents $(u; e_1, e_2)$ is called a {\it
triangle} if  $e_i \in J_0(e_j)$, $i\neq j$, $e_i \in J_2(u)$,
$i=1,2$, $u \in J_1(e_1)\cap J_1(e_2)$, and the following
multiplication rules hold: $P(u)e_i = e_j$, $i\ne j$, and  $P(e_1,
e_2)u = u$. In this case, $e:=e_1+e_2$ is a tripotent such that $e$
and $u$ have the same Peirce spaces. The verification that $(u; e_1,
e_2)$ is a triangle is simplified by the {\it Triangle Criterion\/}
\cite[I.2.5]{Nbook}, which says that {\it as soon as $u$ and $e_1$
are tripotents satisfying $u\in J_1(e_1)$ and $e_1\in J_2(u)$ then
$(u; e_1, P(u)e_1)$ is a triangle.}

A Jordan triple system with a triangle $(u; e_1, e_2)$ is said to be
{\it triangulated} if  $J=J_2(e_1)\oplus \big(J_1(e_1) \cap
J_1(e_2)\big) \oplus J_2(e_2)$ which is equivalent to $J=J_2(e)$. In
this case, we will use the notation $J_i=J_2(e_i)$ and
$M=J_1(e_1)\cap J_1(e_2)$.   Hence
$$J=J_1\oplus M\oplus
J_2.$$ For such a $J$ the index $i$ will always vary in $\{1,2\}$,
in which case $j\in \{1,2\}$ is given by $j=3-i$. An arbitrary
product $P(x)y$ in $J$  has the form
$P(x_1+m+x_2)(y_1+n+y_2)=z_1+r+z_2$, where
\begin{eqnarray*} z_i &=&
P(x_i)y_i+P(m)y_j+\{x_i, n, m\}, \hbox{ and}\\  r &=& P(m)n+\{x_1,
y_1, m\}+ \{x_2, y_2, m\}+\{x_1, n, x_2\}.
\end{eqnarray*}
Using Peirce multiplication rules and standard Jordan identities,
most of these products can be written in terms of the quadratic
operators
$$
       Q_i:  M  \rightarrow  J_i  : m  \mapsto  Q_i(m):=P(m)e_j,
$$
with linearizations $Q_i(m,n)=P(m,n)e_j$, the automorphism
$^-:J\rightarrow J$
$$x\mapsto \overline{x}:=P(e)x,\quad e=e_1+e_2,$$
and the bilinear maps $J_i\times M\rightarrow M$ defined by
\begin{formula}\label{001} $
J_i\times M\rightarrow M :(x_i , m)\mapsto x_i \cdot m
=L(x_i)m:=\{x_i, e_i, m\}. $\end{formula} \noindent Indeed we have
\cite[1.3.2-1.3.6]{MNSimple}:
\begin{formula} \label{01} $P(m)n=Q_i(m,\overline{n})\cdot m -Q_j(m)\cdot
\overline{n}$,
\end{formula}

\begin{formula} \label{02} $\{m,n,x_i\}=Q_i(m,x_i\cdot \overline{n})=\{m, \overline{x_i}\cdot n
,e_i\}$,
\end{formula}

\begin{formula} \label{03} $\{m,x_i, n\}=Q_j(m,\overline{x_i}\cdot n)=Q_j(n,\overline{x_i}\cdot
m)$,
\end{formula}

\begin{formula} \label{04} $\{x_i,y_i,m\}=x_i\cdot (\overline{y_i}\cdot
m)$,
\end{formula}

\begin{formula} \label{05}$\{x_i,m,y_j\}=x_i\cdot (y_j\cdot \overline{m})=y_j\cdot (x_i\cdot \overline{m}),$
\end{formula}
\noindent while $P(x_i)y_i\in J_i$ and $P(m)y_i\in J_j$ cannot be
reduced. Also \cite[1.3.7]{MNSimple}:

\begin{formula} \label{1} $e_i \cdot m=m  \,\,{\rm and}\,\, P(x_i)y_i \cdot m
= x_i\cdot \big(\overline{y_i} \cdot (x_i \cdot m)\big)$,
\end{formula}

Note that $^-$ has period 2 with $\overline{e_i}=e_i$, stabilizes
the Peirce subspaces $J_i$ and $M$ and reduces to $P(e_i)$ on $J_i$
and $P(e_1, e_2)$ on $M$.  We will also consider the {\it square\/}
of elements $x_i\in J_i$ defined as
\begin{formula}\label{eq:squares}
    $x_i^2:=P(x_i)e_i.$
\end{formula}

Because of \ref{1} we have $L(x_i)^2=L(x_i^2)$. We say that $J$ is
{\it faithfully triangulated\/} or that $u$ is {\it faithful\/} if
 any $x_1 \in J_1$ with $x_1\cdot u=0$ vanishes.  We will also
need the traces
$$T_i(m):=Q_i(u,m)=\{u \, e_j \, m\},$$ and the map $${}^*:=P(e)P(u)=P(u)P(e),$$ which
is an automorphism of $J$ of period 2 such that $u^* =u$, $e_i^*
=e_j$, and so $J_i^*=J_j$. The following lemma is shown in the proof
of \cite[1.15]{MNSimple} and will be used later.

\begin{lemma}\label{inheritance} Let $J=J_1\oplus M\oplus J_2$ be a
triangulated Jordan triple system. If $z=z_1+m+z_2 \in \T (J)$, the
set of trivial elements of $J$, then $z_i\in \T (J_i)$ and $m\in \R
Q_i$  for $i=1,2$ where $\R Q_i = \{ m\in M : Q_i(m) = 0 =
Q_i(m,M)\}$.   Conversely, if $m\in \R Q_1\cup \R Q_2$, then
$P(m)M=0=P(P(m)J_i)J$.
\end{lemma}

\begin{definition} \label{quad} If  $M=\bigoplus_{\lambda \in \Lambda}M^\lambda$ and
$N=\bigoplus_{\lambda \in \Lambda}N^\lambda$ are  graded modules and
$Q:M\rightarrow N$ is a quadratic map, we  call $Q$ {\it graded} if
$Q(M^\la) \subseteq N^{2\la}$ and $Q(M^\la, M^\nu) \subseteq N^{\la
+ \nu}$, where $Q(.,.)$ is the bilinear form associated to $Q$.
Recall that the {\it radical} of $Q$ is $\R Q = \{ m\in M : Q(m) = 0
= Q(m,M)\}$. In our situation the {\it graded-radical} of $Q$,
defined as
$${\GR}Q:=\textstyle \bigoplus_{\lambda \in \Lambda}\; \{m \in M^\lambda:
Q(m)=0=Q(m, M)\},$$ will be more important. Naturally, we say that
$Q$ is {\it graded-nondegenerate\/} if ${\GR}Q=0$. It is easily seen
that the submodule $\{ m\in M : Q(m,M)=0\}$ is graded. For any
$m=\sum_{\lambda \in \Lambda}m^\lambda \in M$ satisfying $Q(m,M)=0$,
we have $Q(m)=\sum_{\lambda \in \Lambda} Q(m^\lambda)$, where
$Q(m^\lambda)\in N^{2\la}$. Hence ${\GR}Q =\R Q$ if $\frac{1}{2} \in
k$ or $\La$ does not have $2$-torsion. In general, ${\GR}Q$ is the
greatest graded submodule of $\R Q$.
\end{definition}

\begin{definition} \label{defgratri} A Jordan triple system $J$ is said to be {\it
graded-triangulated} if $J$ is graded by some abelian group $\La$
and triangulated by $(u; e_1, e_2) \subseteq J^0$. We call a
graded-triangulated $J$ {\it $\La$-triangulated} if $\supp_\La J =
\{ \la \in \La : J_i^\la \ne 0 \hbox{ or } M^\la \ne 0\}$ generates
$\La$ as abelian group.   We call $J$ {\it
graded-simple-triangulated} if $J$ is graded-simple and
graded-triangulated.

Let $J$ be a graded-triangulated Jordan triple system. The grading
group of a graded-triangulated Jordan triple system will usually be
denoted by $\La$. Observe that the $\La$-grading is {\it
compatible\/} with the Peirce decomposition: The subsystems $J_i$
and $M$ are  graded:   $J_i=\bigoplus_{\lambda \in
\Lambda}J_i^\lambda$ and $M=\bigoplus_{\lambda \in
\Lambda}M^\lambda$, the quadratic operators $Q_i$ and the
automorphisms ${}^*$ and $^-$ are graded. \end{definition}

As before we let $(u; e_1,e_2)$ be the triangle inducing the
triangulation. The product formulas \ref{01}-\ref{05} show that a
graded linear subspace $K=K_1\oplus N\oplus K_2$ with $K_i\subseteq
J_i$, $N\subseteq M$, is a graded subsystem if
\begin{formula}\label{sub}$\overline{K_i}=K_i,\,\overline{N}=N,\,
P(K_i)K_i\subseteq K_i,\, P(N)K_i\subseteq K_j\, \hbox{ and }
K_i\cdot N\subseteq N$.\end{formula}

As in \cite{MNSimple}, we denote by $C$ the subalgebra of End$_k(M)$
generated by
$$C_0=L(J_1)=\textstyle \bigoplus_{\la \in \La}L(J_1^\la),$$ and we
say that $u$ is {\it $C$-faithful} if $c u=0$ implies $c=0$.  It is
easily seen that
$\End^\La_k(M) :=\bigoplus_{\la \in \La}\End_k^\la(M)$, 
where $\End_k^\la(M)=\{\varphi\in \,\End_k(M):
\varphi(M^\ga)\subseteq M^{\ga+\la} \text{ for all } \ga \in \La\}$,
is a subalgebra of $\End_k M$ which is  graded by $\La$. Note that
$L(J_i^\la) \in$ End$_k^\la(M)$. Hence $C_0$ is a graded submodule,
and this implies that $C$ is a graded   subalgebra of
$\End^\La_k(M)$, i.e., $C=\bigoplus_{\la \in \La}C^\la$ where
$C^\la=C\cap \End_k^\la(M)$.   We have that
$\overline{L(x_1)}:=L(\overline{x_1})=P(e)L(x_1)P(e)\in C_0$.
Therefore $c\mapsto \overline{c}=P(e)cP(e)|_M$ defines an
automorphism on $C$ of period $2$, which is graded. Moreover, $L:
J_1 \to C_0$, $x_1\mapsto L(x_1)$ is a nonzero graded specialization
with respect to $P(c_0)d_0=c_0\overline{d_0}c_0\in C_0$, for $c_0,
d_0 \in C_0$. By \cite[1.6.6]{MNSimple}, $C$ has a {\it reversal
involution\/} $\pi$, i.e.,
\begin{formula}\label{pi}
$\big(L(x_1) \cdots L(x_n)\big)^\pi=L(x_n) \cdots
L(x_1).$\end{formula}

\noindent It easily follows {f}rom \ref{03} that
\begin{formula}\label{pi2} $Q_2(cm,n)=Q_2(m,c^\pi  n)$.\end{formula}

\noindent It is clear {f}rom \ref{pi} that $\pi$ is homogeneous of
degree $0$ and
 commutes with the automorphism $^-$ of $C$. Moreover,
 $C_0=\overline{C_0}\subseteq H(C,\pi )$ is an {\it ample subspace} of
$(C,\pi)$, i.e., $1\in C_0$ and $cC_0c^\pi\subseteq C_0$ for all
$c\in C$. Indeed, for $c\in C$, $x_1\in J_1$, we have by
\cite[1.6.9]{MNSimple}

\begin{formula} \label{2} $cL(x_1)c^\pi =L(P(cu)P(u)x_1)$. \end{formula}

\noindent Besides the formulas already mentioned we will use the
following identities proven in \cite[1.6.8, 1.6.11, 1.6.2, 1.6.3,
1.6.10, 1.6.12, 1.6.14]{MNSimple}. For $c\in C$, $m\in M$ and
$x_i\in J_i$ we have

\begin{formula} \label{3} $c+c^\pi =L\big(T_1(cu)\big)$,
\end{formula}
\begin{formula} \label{3.5} $Q_2(cu,m)=T_2(c^\pi m)$,
\end{formula}
\begin{formula}\label{4} $T_i(m)^*=T_j(m^*)=T_j(m)$ and $Q_i(m)^*=Q_j(m^*)$, \end{formula}
\begin{formula} \label{5} $m^*=T_i(m)\cdot u -m$, \end{formula}
\begin{formula}\label{6} $(cu)^* =c^* u=c^{\pi} u$  {\rm and
hence} $cc^*u=c^*cu$, {\rm where $c^*=P(u)cP(u)$.

\noindent Note that $C^*$ is the subalgebra of End$_k(M)$ generated
by $L(J_2)$.}\end{formula}
\begin{formula}\label{7} $(x_i^* -x_i) \cdot m=\big(T_i(m)\cdot x_i -T_i(x_i \cdot m)\big)\cdot u
= -\Gamma_i (x_i;m) u,$\end{formula} \noindent where $\Gamma_i
(x_i;m):=L\big(T_i(x_i\cdot m)\big)-L\big(T_i(m)\big)L(x_i)$.
Observe that $\Gamma_i (x_i;m)$ is linear in the two variables.
\begin{formula}\label{8}
\begin{multline*}
\Gamma_1 (x_1;m)\Gamma_1 (x_1;m)^\pi m  =
L\big(Q_2(m)\big)[L(x_1),\Gamma_1(x_1;m)]u \\
+  L(x_1)[L\big(Q_1(m)\big), L(x_1)]m   +   [L(x_1),
L(P(m)P(u)x_1)]m \in Cu.
\end{multline*}
\end{formula} \sm

The following proposition is a straightforward generalization of the
corresponding result for $\La=0$
 \cite[1.15]{MNSimple}. Its proof, which uses Prop.~\ref{nonde}, Th.~\ref{grsiminh} and Lem.~\ref{inheritance}, will be left to the reader.

\begin{proposition}\label{peirce} Let $J$ be a graded-triangulated
Jordan triple system. Then
\begin{enumerate}
 \item[{\rm (i)}] $J$ is graded-nondegenerate iff $J_1$ and
 $Q_1$ are graded-nondegenerate, iff $J_2$ and $Q_2$
are graded-nondegenerate. In this case, $J$ is faithfully
triangulated.
\item[{\rm (ii)}] $J$ is graded-simple iff $J_1$ is graded-simple
and $Q_1$ is graded-nondegene\-ra\-te, iff $J_2$ is graded-simple
and $Q_2$ is graded-nondegenerate.
\end{enumerate} \end{proposition}

\begin{definition} A  graded-triangulated Jordan triple system $J$ is
called {\it division-triangulated} if the Jordan triple systems
$J_i$, $i=1,2$, are division-graded and if every homogeneous $0\neq
m\in M$ is invertible in $M$, equivalently in $J$. We call $J$ {\it
division-$\La$-triangulated\/} if $J$ is division-triangulated and
$\La$-triangulated. Thus, $\supp_\La J = \{ \la \in \La : J_i^\la
\ne 0 \hbox{ or } M^\la \ne 0\}$ generates $\La$ as abelian group.

Any division-triangulated Jordan triple system is in particular
graded-simple. Since $e_1+e_2$ is invertible, an off-diagonal
element $m\in M$ is invertible iff $P(m)(e_1 + e_2) = Q_1(m) \oplus
Q_2(m)$ is invertible in $J$ which is equivalent to both $Q_i(m)$
being invertible in $J_i$.

Let $k$ be a field. A {\it triangulated Jordan triple torus} is a
division-triangu\-la\-ted Jordan triple system $J=J_1 \oplus M\oplus
J_2$ for which $\dim_k J^\la_i \le 1 $ and $\dim_k M^\la \le 1$ for
all $\la \in \La$. We call such a Jordan triple system a {\it
$\La$-triangulated Jordan triple torus\/} if $J$ is
division-$\La$-triangulated.

We will use the same approach to define ``tori'' and $\La$-tori in
other categories: associative algebras (\ref{asstordef}), Jordan
algebras (\ref{jordalgtor}), Jordan pairs (\ref{jptordef}) and Lie
algebras (\ref{lietordef}), and we will see in \S\ref{b2lie} the
connection between them:  $\La$-triangulated Jordan structures
coordinatize $\rmb_2$-Lie tori, which is our principal motivation
for studying them.

\end{definition}

For the next lemma we recall that a subset $S\subset \La$ is called
a {\it pointed reflection subspace\/} if $0\in S$ and $2S-S\subset
S$, see for example \cite[2.1]{NY}, where it is also shown that any
pointed reflection subspace is a union of cosets modulo $2\ZZ[S]$,
including the trivial coset $2\ZZ[S]$. Here $\ZZ[S]$ denotes the
$\ZZ$-span of $S$. In particular, a pointed reflection subspace is
in general not a subgroup.

\begin{lemma} \label{supplem}
Let $J$ be a division-$\La$-triangulated Jordan triple system . Put
$$
      \scL= \supp_\La J_1 = \supp_\La J_2 \quad\text{and} \quad
       \scS= \supp_\La M.
$$
Then $\scL$ and $\scS$ are pointed reflection subspaces of $\La$
satisfying \begin{equation} \scL + 2 \scS \subset \scL
\quad\text{and}\quad \scL + \scS \subset \scS. \label{eq:supplem1}
\end{equation}
\end{lemma}

\begin{proof} We have $\supp_\La J_1 = \supp_\La J_2$ by applying
the invertible operator $P(u)$ to $J_i$. That $\scL$ and $\scS$ are
pointed reflection spaces is a general fact which is true for any
division-graded Jordan triple system $J= \bigoplus_{\la \in \La}
J^\la$ with $J^0\ne 0$: We have, with obvious notation,
$(y^\mu)^{-1}\in J^{-\mu}$, whence $P(x^\la)(y^\mu)^{-1} \in J^{2\la
- \mu}$. The formulas in (\ref{eq:supplem1}) follow from $P(m^\la)
x_1^\mu \in J_2^{2\la +\mu}$ and invertibility of $L(x_1^\la)$ on
$M$. \end{proof}

\begin{remark} The relations (\ref{eq:supplem1}) are well-known from
the theory of extended affine root systems of type $\rmb_2$
(\cite[II]{aabgp}), or more generally, the theory of extension data
for affine reflection systems (\cite{prs}). This is of course no
accident in view of the connections between triangulated Jordan
structures and $\rmb_2$-graded Lie algebras, explained in
\S\ref{b2lie}. \end{remark}

\section{Hermitian matrix systems} \label{sec:trhms}

In this section we introduce the first of the two basic models for
our paper, the hermitian matrix system (Def.~\ref{examplehermi}),
and we characterize them within the class of all triangulated Jordan
triple systems in Th.~\ref{hermitian}. We then describe the graded
ideals of a  hermitian matrix system (Prop.~\ref{relideals}), which
allows us to describe the graded-(semi)prime and graded-simple
hermitian matrix systems (Cor.~\ref{hermitiansimpl} and
Prop.~\ref{hermconclusion}). Finally, in Lem.~\ref{hermtor} we
describe the division-triangulated  and tori among the hermitian
matrix systems. \sm

\begin{definition} \label{examplehermi} {\it Hermitian matrix
systems $\rmH_2(A,A_0 ,\pi , {}^-)$.}  Recall \cite[\S2]{MNSimple}
that an (associative) coordinate system $(A,A_0 ,\pi , {}^-)$
consists of a unital associative $k$-algebra $A$ with involution
$\pi$ and an automorphism $^-$ of period $2$ commuting with $\pi$,
together with a $^-$ stable {\it $\pi$-ample} subspace $A_0$, i.e.,
$\overline{A_0}=A_0 \subseteq \rmH(A,\pi)$, $1\in A_0$ and
$aa_0a^\pi \subseteq A_0$ for all $a\in A$ and $a_0\in A_0$. We will
call such a coordinate system  {\it graded} by $\La$ if
$A=\bigoplus_{\la \in \La}A^\la$ is graded, $\pi$ and $^-$ are
homogeneous of degree $0$ and $A_0$ is a graded submodule:
$A_0=\bigoplus_{\la \in \La}A_0^\la$ for $A_0^\la=A_0\cap A^\la$.

To a  graded  coordinate system $(A,A_0,\pi, ^-)$ we associate the
{\it  hermitian matrix system} $\rmH=\rmH_2(A,A_0,\pi , ^-)$ which
by definition is the Jordan triple system of  $2\times 2$-matrices
over $A$ which are hermitian $(X=X^{\pi t})$ and have diagonal
entries in $A_0$, with triple product $P(X)Y=X\overline{Y}^{\pi
t}X=X\overline{Y}X$, $t={\rm transpose}$. The Jordan triple system
$\rmH_2(A,A_0,\pi , ^-)$ is spanned by elements
$$a_0^\la [ii]=a_0^\la E_{ii} \quad\hbox{ and } \quad
a^\ga [12]=a^\ga E_{12}+(a^\ga)^\pi E_{21}=(a^\ga)^\pi[21],$$
$a^\ga\in A^\ga, a_0^\la\in A_0^\la$. Such a system is  graded by
$\La$: $\rmH=\bigoplus_{\la\in \La}\rmH^\la$ where $H^\la ={\rm
span}\{a_0^\la [ii], a^\la [12]:\, a_0^\la\in A_0^\la , a^\la\in
A^\la \}$, and graded-triangulated by ($u=1[12]$; $e_1=1[11]$,
$e_2=1[22])\in \rmH^0$.  Note that the automorphisms $^-$ and $^*$
of $\rmH$ defined in \S\ref{sec:gr-triang JTS general} are
\begin{eqnarray*} \overline{a_0[11]+a[12]+b_0[22]} &=&
        \overline{a_0}[11]+\overline{a}[12]+\overline{b_0}[22],\\
 (a_0[11]+a[12]+b_0[22])^* &=&
 b_0[11]+a^\pi[12]+a_0[22].\end{eqnarray*}
We say that $\rmH$ is {\it diagonal} if the diagonal coordinates
$A_0$ generate all coordinates $A$. In this case, the involution
$\pi$ is the reversal involution with respect to $A_0$, i.e.,
$\pi(a_1 \cdots a_n) = a_n \cdots a_1$ for $a_i \in A_0$.
\end{definition}

\begin{example} \label{hermexam}
{\rm  As an example, suppose $A=B\boxplus B^{\rm op}$ is a direct
algebra sum of an associative  graded  algebra $B$ and its opposite
algebra $B^{\rm op}$ and that $\pi$ is the exchange involution
$(b_1, b_2) \mapsto (b_2, b_1)$ of $A$. (Here and in the following
$\boxplus$ denotes the direct sum of ideals.) Then necessarily $A_0
= \{(b,b) : b\in B\}$ and $\rmH_2(A,A_0,\pi, ^-)$ is canonically
isomorphic to $\Mat_2(B)$, the $2\times 2$-matrices over $B$, with
Jordan triple product $P_xy= x\bar yx$ or $P_xy=x\bar y^tx$
depending on the automorphism $^-$ of $A$. Namely, we have the first
case if $\overline{(b_1, b_2)} =  (\bar b_1, \bar b_2)$ where $b
\mapsto \bar b$ is an automorphism of $B$ of period $2$, and we have
the second case if $\overline{(b_1, b_2)} = (b_2^\iota, b_1^\iota)$
where $b\mapsto b^\iota$ is an involution of $B$. }\end{example}

Within triangulated Jordan triple systems, the hermitian matrix
systems can be characterized as follows.

\begin{theorem} \label{hermitian} {\rm \textsc{triangulated hermitian
coordinatization theorem}}. {\rm (\cite[2.4]{MNSimple} for $\La=0$)}
For any graded  Jordan triple system $J=J_1\oplus M\oplus J_2$ which
is faithfully triangulated by $(u; e_1, e_2)$, the graded subsystem
$$J_h=J_1\oplus Cu\oplus J_2,$$ where $C$ denotes the subalgebra
of\/ $\End_k(M)$ generated by $C_0=L(J_1)$, is graded-triangulated
by $(u;e_1,e_2)$ and graded isomorphic to the diagonal
 hermitian matrix system $\rmH=\rmH_2(A,A_0 ,\pi
, ^-)$ under the map
$$x_1\oplus cu\oplus x_2\mapsto
 \left( \begin{array}{cr} L(x_1) & c \\  c^\pi & L(x_2^*)\end{array}
 \right)$$
for $A=C|_{Cu}$, $A_0=C_0|_{Cu}$, $c^\pi$ as in \ref{pi}, and
$\overline{c}=P(e)\circ c \circ P(e)|_{Cu}$. The above isomorphism
maps the triangle $(u; e_1, e_2)$ of $J$ onto the standard triangle
$(1[12]; 1[11], 1[22])$ of\/ $\rmH$. We have $J=J_h$ as graded
triple systems if and only if $M=Cu$.
\end{theorem}
\begin{proof}
If $J$ is a graded-triangulated Jordan triple system, then
$(A,A_0,\pi , ^-)$, for $A=C|_{Cu}$, $A_0=C_0|_{Cu}$,
 $d^\pi$ as in \ref{pi} and
$\overline{d}=P(e)\circ d \circ P(e)$, is a  graded coordinate
system. Since $A_0$ generates $A$, $\rmH_2(A,A_0,\pi , ^-)$ is a
diagonal hermitian matrix system. Also, by definition, $J_h$ is
graded.

Now it follows {f}rom \cite[2.4]{MNSimple} that $J_h$ is a subsystem
of $J$ isomorphic to $\rmH_2(A,A_0,\pi , ^-)$ under the map
$$x_1\oplus cu\oplus x_2\mapsto
 \left( \begin{array}{cr} L(x_1) & c \\  c^\pi & L(x_2^*)\end{array}
 \right),$$ which is clearly a graded isomorphism (recall that $(J_i^\la)^* = J_j^\la$).
 That the isomorphism
preserves the triangles is clear. Also by \cite[2.4]{MNSimple}, we
have that $J_h=J$ if and only if $M=Cu$.
\end{proof}

 Let $A$ be  a graded algebra and let $\frS$ be a set of endomorphisms
of $A$ preserving the grading. We call $I$ a {\it graded ideal of
$(A,\frS)$\/} if $I$ is a graded ideal left invariant by all $s\in
\frS$. If $\scP$ is a property of an algebra defined in terms of
ideals we will say that $(A,\frS)$ is {\it graded-$\scP$\/} if
$\scP$ holds for all graded ideals of $(A,\frS)$. We will apply this
for $\scP$=graded-(semi)prime and $\scP$=graded-simple. For
$\frS=\{\pi, ^-\}$ as above we will determine the graded-simple
$(A,\frS)$-structures in Prop.~\ref{relsimple}. Here we only note:

\begin{remark} \label{semip} {\it If $\frS$ is a finite semi-group consisting of automorphisms or involutions of a
graded associative algebra $A$, then $(A,\frS)$ is graded-semiprime
if and only if $A$ is graded-semiprime.} Indeed, if $I$ is a graded
ideal of $A$ with $I^2=0$ then $\hat I = \sum_{s\in \frS}\, s(I)$ is
an $\frS$-invariant graded ideal of $A$ with ${\hat I}^n=0$ for $n>
|\frS|$. Hence $\hat I =0$ and so also $I=0$.
\end{remark}

\begin{proposition}\label{relideals} Let $\rmH=\rmH_2(A,A_0,\pi , ^-)$
be a  hermitian matrix system. Then the graded ideals of\/ $\rmH$
are exactly the submodules
$$\rmH_2(B, B_0)=B_0[11]\oplus B[12]\oplus B_0[22]$$
for $(\pi , ^-)$-invariant graded submodules $B_0\subseteq A_0$ and
$B\subseteq A$ such that for $a\in A$, $a_0\in A_0$, $b\in B$, and
$b_0\in B_0$,
\begin{enumerate}
\item[{\rm (1)}] $ba+b^\pi a^\pi, \, ba_0b^\pi$, and $ab_0a^\pi$ lie
in $B_0$,
\item[{\rm (2)}] $ab_0,\, a_0b,\, aba,$ and $bab$ lie in $B$.
\end{enumerate}
In particular,
\begin{enumerate}
\item[{\rm (i)}] if $B$ is a graded ideal of $(A,\pi, ^-)$,  then $\rmH_2(B, B\cap
A_0)$ is a graded ideal of\/ $\rmH_2(A,A_0, \pi , ^-)$, and,
conversely,
\item[{\rm (ii)}] if $(A,\pi, ^-)$ is graded-semiprime and\/ $\rmH_2(B, B_0)$ is a
nonzero graded ideal of\/ $\rmH_2(A,A_0, \pi , ^-)$, then there
exists a nonzero graded ideal $I$ of $(A,\pi, ^-)$ such that
$\rmH_2(I, I_0)\subseteq \rmH_2(B, B_0)$, for $I_0=I\cap B_0$.
\end{enumerate}\end{proposition}

\begin{proof} This easily follows {f}rom the case $\La=0$ which is
proven in \cite[2.7]{MNSimple}. \end{proof}

As a consequence, we have the following corollary whose proof is
again omitted since it is based on a standard argument.

\begin{corollary} \label{hermitiansimpl} {\rm (\cite[2.7(5), 2.11]
{MNSimple} for $\La=0$)} Let $\rmH=\rmH_2(A,A_0, \pi , ^-)$ be a
 hermitian matrix system. Then
\begin{enumerate}
 \item[{\rm (i)}] $\rmH$ is graded-simple iff $(A, \pi , ^-)$ is graded-simple.
\sm
\item[{\rm (ii)}] The following are equivalent:
 \begin{enumerate}
   \item[{\rm (a)}] $\rmH$ is graded-nondegenerate,

    \item[{\rm (b)}] $\rmH$ is graded-semiprime,

    \item[{\rm (c)}] $(A, \pi , ^-)$ is graded-semiprime,

    \item[{\rm (d)}] $A$ is graded-semiprime.
  \end{enumerate}
\sm

\item[{\rm (iii)}]$\rmH$ is graded-prime iff $(A, \pi , ^-)$ is graded-prime.
\end{enumerate}\end{corollary}

Because of Cor.~\ref{hermitiansimpl}(i) it is of interest to
determine the graded-simple coordinate systems $(A,\pi, ^-)$. This
can be done without assuming that $A$ is associative.

\begin{proposition} \label{relsimple}
Let $A$ be an arbitrary, not necessarily associative,  graded
 algebra  with commuting involution $\pi$ and automorphism $^-$
of order $2$, which are homogeneous of degree $0$. Then the
graded-simple structures $(A,\pi, ^-)$ are precisely the following:
 \begin{enumerate}
\item[{\rm (I)}] graded-simple $A$ with graded involution $\pi$ and graded automorphism
$^ -$;

\item[{\rm (II)}]  $A\cong_{\La}
B\boxplus B^{\rm op}$ with exchange involution $\pi$ for a
graded-simple $B$ with graded automorphism $^- : (b_1, b_2)^\pi=
(b_2, b_1), \overline{(b_1, b_2)}=(\overline{b_1},\overline{b_2})$;

\item[{\rm (III)}] $A\cong_{\La} B\boxplus B^{\rm
op}$ with exchange involution $\pi$ for a graded-simple $B$ with
graded involution $\iota : (b_1,b_2)^\pi=(b_2,b_1),
\overline{(b_1,b_2)}=({b_2}^\iota,{b_1}^\iota)$;

\item[{\rm (IV)}] $A\cong_{\La} B\boxplus B$ with exchange
 automorphism $^-$ for a graded-simple $B$ with graded involution $\pi :
(b_1,b_2)^\pi=({b_1}^\pi,{b_2}^\pi),
\overline{(b_1,b_2)}=(b_2,b_1)$;

\item[{\rm (V)}] $A\cong_{\La} B\boxplus B^{\rm op} \boxplus B\boxplus
B^{\rm op}$ for a graded-simple $B$ with $\pi$ the exchange
involution of $C= B\boxplus B^{\rm op}$ and $^-$ the exchange
automorphisms of $C\boxplus \overline{C}:
(a_1,a_2,a_3,a_4)^\pi=(a_2,a_1,a_4,a_3)$ and $
\overline{(a_1,a_2,a_3,a_4)}=(a_3,a_4,a_1,a_2)$.
\end{enumerate}
\end{proposition}
\begin{proof} The proof is again a straightforward generalization of the
corresponding result in the ungraded situation, which is
\cite[2.8]{MNSimple}. \end{proof}

For later use we note the following special case of
Prop.~\ref{relsimple} for a commutative algebra $D$ and $\pi$ the
identity ``involution''. We note that $D$ is not assumed to be
unital.

\begin{corollary}\label{patrick} Let $D$ be a commutative  graded
 algebra with a graded automorphism $^-$ of order $2$. Then
$(D, ^-)$ is graded-simple if and only if either $D$ is
graded-simple or $D\cong_{\La} B\boxplus B$, for a commutative
graded-simple $B$ with the exchange automorphism.\end{corollary}

Recall (\cite[\S1.14]{JP}) that a Jordan triple system $T$ is called
{\it polarized} if there exist submodules $T^\pm$ such that
$T=T^+\oplus T^-$ and for $\sigma =\pm$ we have
$P(T^\sigma)T^{\sigma}=0=\{T^\sigma,T^\sigma,T^{-\sigma}\}$ and
$P(T^\sigma)T^{-\sigma}\subseteq T^\sigma$. In this case,
$V=(T^+,T^-)$ is a Jordan pair. Conversely, to any Jordan pair
$V=(V^+,V^-)$ we can associate a polarized Jordan triple system
$T(V)=V^+\oplus V^-$ with quadratic map $P$ defined by
$P(x)y=Q(x^+)y^-\oplus Q(x^-)y^+$
 for $x=x^+\oplus x^-$ and $y=y^+\oplus y^-$. In fact, the
category of Jordan pairs is equivalent to the category of polarized
Jordan triple systems. It is also known that for any Jordan triple
system $T$ the pair $(T,T)$ is a Jordan pair (\cite[\S1.13]{JP}).
Hence, it has an associated polarized Jordan triple system which we
will denote $T\oplus T$. Examples are the cases (IV) and (V) of
Prop.~\ref{hermconclusion} below.

\begin{proposition}\label{hermconclusion} {\rm \textsc{hermitian
graded-simplicity criterion}}. A  graded  Jordan triple system is a
graded-simple-triangulated hermitian matrix system $\rmH_2(A,A_0,\pi
, ^-)$ if and only if it is graded isomorphic to one of the
following:
\begin{enumerate}
\item[{\rm (I)}] $\rmH_2(A,A_0, \pi , ^-)$ for a graded-simple $A$;

\item[{\rm (II)}] $\Mat_2(B)$ for a graded-simple associative unital $B$ with graded automorphism
$^-$, where $\overline{(b_{ij})}=(\overline{b_{ij}})$ for
$(b_{ij})\in \Mat_2(B)$ and $P(x)y=x\overline{y}x$;

\item[{\rm (III)}] $\Mat_2(B)$ for a graded-simple associative unital $B$ with graded involution
$\iota$,  where $\overline{(b_{ij})}=(b_{ij}^\iota)$ for
$(b_{ij})\in \Mat_2(B)$ and $P(x)y=x\overline{y}^tx$;

\item[{\rm (IV)}] polarized $\rmH_2(B,B_0, \pi)\oplus \rmH_2(B,B_0, \pi)$
for a graded-simple  $B$ with graded involution $\pi$;

\item[{\rm (V)}] polarized $\Mat_2(B)\oplus \Mat_2(B)$ for a
graded-simple associative unital $B$ and $P(x)y=xyx$.\end{enumerate}
Among the cases {\rm (II)--(V)}, the matrix system is diagonal iff
$B$ is noncommutative.
\end{proposition}
\begin{proof} By definition  and Cor.~\ref{hermitiansimpl}(i), a graded
Jordan triple system is a graded-simple $J=\rmH_2(A,A_0, \pi , ^-)$
if and only if $(A,A_0, \pi , ^-)$ is a  graded  coordinate system
where $(A, \pi , ^-)$ is graded-simple. Since the graded-simple
structures $(A, \pi , ^-)$  have been described in
Prop.~\ref{relsimple}, it now suffices to show that the cases
(I)--(V) of Prop~\ref{relsimple} correspond to the cases (I)--(V)
above. This is straightforward and will be left to the reader, see
\cite[2.10]{MNSimple} for the case
 in which the grading group $\La=0$. \end{proof}

In order  to describe division-triangulated hermitian matrix systems
we need to introduce some concepts from the theory of
division-graded algebras.

\begin{definition} \label{asstordef} A unital associative  graded algebra $A=\bigoplus_{\la \in \La} A^\la$ is called {\it
predivision-graded} if every nonzero homogeneous space contains an
invertible element. The {\it support\/} $\supp_\La A = \{ \la\in \La
: A^\la \ne 0\}$ of a predivision-graded $A$ is a subgroup of $\La$.
We will call $A$ {\it predivision-$\La$-graded\/} if $\supp_\La A =
\La$.

After choosing a family of invertible elements $(u_\la: \la \in
\La)$ with $u_\la \in A^\la$, one can identify a
predivision-$\La$-graded algebra $A$ with a crossed-product algebra
$A=(B,\La, \sigma,\tau)$ in the sense of \cite{Pass} with $B=A^0$
and the twist $\tau$ and the action $\sigma$ defined by $u_\la u_\mu
= \tau(\la,\mu) u_{\la + \mu}$ and $u_\la b=
\big({^{\sigma(\la)}}(b)\big)u_\la$ for $b\in B$.

An example of a predivision-$\La$-graded algebra is the so-called
{\it twisted group algebra $B^t[\La]$\/}, i.e., the crossed product
algebra $(B,\La, \sigma,\tau)$ with $\sigma(\la) = \Id_B$ for all
$\la \in \La$. An immediate special case of a twisted group algebra
is $k[\La]$, the {\it group algebra} of $\La$ over $k$ where
$\tau(\la,\mu) = 1_k$ for all $\la,\mu \in k$.

A unital associative graded algebra $A$ is called {\it
division-graded\/} if every non\-zero homogeneous element is
invertible, and  such an algebra  is called {\it
division-$\La$-graded\/} if $\supp_\La A= \La$. A
division-$\La$-graded algebra $A$ is the same as a crossed product
algebra  $(B,\La, \sigma,\tau)$ with $B$ a division algebra. A
division-graded algebra is in particular graded-simple.

A unital associative commutative graded algebra $A$ is graded-simple
if and only if it is division-graded. Such algebras will be called
{\it graded-fields}, more precisely {\it $\La$-graded-fields\/} if
$\supp_\La A = \La$. A $\La$-graded-field is the same as a twisted
group algebra $B^t[\La]$ with $B$ a field. If $\La$ is free, a
$\La$-graded field is isomorphic to the group algebra of $\La$.

If $A$ is a division-graded algebra defined over a field $k$ and
such that $\dim_k A^\lambda \le 1$, then $A$ is said to be an {\it
(associative) torus\/}. In this case we call $A$ a {\it
$\La$-torus\/} if $\supp_\La A= \La$. From the point of view of
crossed product algebras, a $\La$-torus is the same as a twisted
group algebra $k^t[\La]$ over the field $k$. We note that in this
case $\tau$ is a $2$-cocycle of $\La$ with coefficients in $k$.
 \end{definition}

\begin{example} $\ZZ^n$-tori. {\rm  Let $A$ be a $\ZZ^n$-torus and choose
nonzero $t_i \in A^{\e_i}$, where $\e_i$ is the ith-canonical basis
vector of $\ZZ^n$. Then the algebra structure of $A$ is uniquely
determined by the rules \begin{equation}\label{eq:quantor1}
  t_i t_i^{-1} = 1_A = t_i^{-1} t_i, \;1\le i \le n\quad\hbox{and}\quad
   t_it_j = q_{ij} t_j t_i, \; 1\le i,j\le n \end{equation}
where $q_{ij}\in k$ satisfy \begin{equation} \label{eq:quantor2}
   q_{ii} = 1 = q_{ij}\,q_{ji} \quad\hbox{for } 1\le i,j\le n.
\end{equation}
For example $A^\la = k t^\la$ for $\la=(\la_1, \ldots, \la_n) \in
\ZZ^n$ and $t^\la= t_1^{\la_1}\, t_2^{\la_2} \cdots t_n^{\la_n}$.
Conversely, let $q=(q_{ij})$ be a $n\times n$-matrix over the field
$k$ whose entries satisfy (\ref{eq:quantor2}), then the associative
unital algebra $k_q$ defined by generators $t_i, t_i^{-1}$ and
relations (\ref{eq:quantor1}) is a $\ZZ^n$-torus. It is customary to
call $k_q$ a {\it quantum $\ZZ^n$-torus\/} or simply a {\it quantum
torus\/} if the grading group is not important, since $k_q$ can be
viewed as a quantization of the coordinate ring of the $n$-torus
$(k^\times)^n$, i.e. the Laurent polynomial ring in $n$ variables.
Observe that $k_q$ is a Laurent polynomial ring iff all $q_{ij}=1$.

A quantum torus has a graded involution iff all $q_{ij} = \pm 1$
(\cite[\S2]{AG}). In this case, an example of a well-defined
involution is the {\it reversal involution\/} $\pi_{\rm rev}$ with
respect to the generating set $\{ t_1^{\pm 1}, \ldots, t_n^{\pm
1}\}$:
\begin{equation}\label{eq:quantor3}
   \pi_{\rm rev}\big(t_1^{\la_1}\, t_2^{\la_2} \cdots t_n^{\la_n}\big) = t_n^{\la_n} t_{n-1}^{\la_{n-1}}
                \cdots t_1^{\la_1}
\end{equation}
}\end{example}

The following lemma is immediate from the definitions above and the
multiplication rules of hermitian matrix systems.

\begin{lemma} \label{hermtor} Let $\rmH=\rmH_2(A,A_0,\pi,{}^-)$ be a
hermitian matrix system.

{\rm (a)} The following are equivalent: \begin{enumerate}

\item[{\rm (i)}] Every homogeneous $0\neq m\in M=A[12]$ is invertible
in $\rmH$,

\item[{\rm (ii)}] $A$ is division-graded,

\item[{\rm (iii)}] $\rmH$ is division-triangulated.
\end{enumerate} \sm

{\rm (b)} Let $k$ be a field. Then $\rmH$ is a $\La$-triangulated
Jordan triple torus iff $A$ is a $\La$-torus.
\end{lemma}

\begin{proof} The implication (i) $\Rightarrow$ (ii) follows from
the multiplication rules of $H$. If $A$ is division-graded, a
nonzero homogeneous element $a_0\in A_0$ is invertible in $A$, say
with inverse $b_0$. We have $b_0\in \rmH(A,\pi)$ since $a_0\in
\rmH(A,\pi)$. But then $b_0 = b_0 a_0 b_0^\pi \in A_0$, whence $A_0$
is division-graded, proving (iii). The implication (iii)
$\Rightarrow$ (i) is immediate, and (b) follows from (a).
\end{proof}

\begin{example}\label{hermtorqun} {\rm As a special case of Lem.~\ref{hermtor}
we get: $\rmH_2(A, A_0, \pi, {}^-)$ is a $\ZZ^n$-triangulated Jordan
triple torus iff $A$ is a quantum $\ZZ^n$-torus.}\end{example}


\section{Clifford systems} \label{sec:triclis}

In this section we introduce the second model of a
graded-triangulated Jordan triple system, the ample Clifford systems
(Def.~\ref{cli}), and we characterize them within the class of
triangulated Jordan triple systems in Th.~\ref{clifford}. We
describe the graded-(semi)prime, graded-strongly prime and
graded-simple Clifford systems in Prop.~\ref{cliffordsimpl} and the
division-triangulated and tori among the Clifford systems in
Cor.~\ref{clifforddivcor}. \sm

\begin{definition} \label{defqfgen} {\it Quadratic form triples.}
 Let $D=\bigoplus_{\la \in \La}D^\la$ be a
graded  unital commutative associative $k$-algebra endowed with an
involution $^-$ of degree $0$, i.e., $\overline{D^\la}=D^\la$ for
all $\la \in \La$. If
\begin{enumerate}
\item[{\rm (i)}] $V$ is a  graded  $D$-module, i.e., $V=\bigoplus_{\la \in \La}V^\la$
is a decomposition into $k$-submodules such that $d^\lambda  x^\ga
\in V^{\la +\ga}$ for $d^\la\in D^\la$, $x^\ga\in V^\ga$ and all
$\la,\ga\in \La$,
\item[{\rm (ii)}] $q:V\rightarrow
D$ is a graded $D$-quadratic form (cf. Def.~\ref{quad}), and
\item[{\rm (iii)}] $S:V\rightarrow V$ is a
hermitian isometry of order $2$ and degree $0$, i.e.,
$S(dx)=\overline{d}S(x)$ for $d\in D$,
$q\big(S(x)\big)=\overline{q(x)}$, $S^2=\Id$ and $S(V^\la) = V^\la$,
\end{enumerate}
then $V$ becomes a Jordan triple system, denoted $J(q,S)$ and called
a {\it quadratic form triple}, by defining $P(x)y=q\big(x,
S(y)\big)x-q(x)S(y) $ for $x,y \in V$ (see for example \cite[\S1,
Ex.~1.6]{Nbook}). Clearly $J(q,S)$ is  graded by $\La$.  We note for
later use:
\begin{equation}\label{eq:qfinvert}
\text{\it $x\in J(q,S)$ is invertible $\iff q(x) \in D$ is
invertible, }
\end{equation}
and then $x^{-1} = q(x)^{-1} S(x)$.
\end{definition}

\begin{definition}\label{cli} {\it Ample Clifford
systems $\AC(q, S, D_0)$.} We consider $(M,q,\allowbreak S,u)$,
where $(M,q,S)$ satisfy (i)--(iii) of Def.~\ref{defqfgen} above and
in addition
\begin{enumerate}
\item[{\rm (iv)}] there exists $u\in M^0$ with $q(u)=1$ and $S(u)=u$.
\end{enumerate}
We then define $(\tilde M, \tilde{q},\tilde{S})$ as follows:
\begin{enumerate}
\item[{\rm (i)}$'$] $\tilde M :=De_1\oplus M\oplus D e_2$, where $De_1 \op
De_2$ is a free graded $D$-module with basis $(e_1, e_2)$ of degree
$0$,

\item[{\rm (ii)}$'$] ${\tilde q} : \tilde M \to D$ is the quadratic form given by
$\tilde{q}(d_1e_1\oplus m \oplus d_2 e_2)=d_1d_2 -q(m)$, whence
$De_1 \op De_2$ is a hyperbolic plane orthogonal to $M$, and

\item[{\rm (iii)}$'$] $\tilde{S}:\tilde M\rightarrow \tilde M$ is the map $
d_1e_1\oplus m \oplus d_2 e_2 \mapsto \overline{d_2}e_1\oplus -S(m)
\oplus \overline{d_1}e_2$.
\end{enumerate}

It is then easily checked that $(\tilde M, \tilde q, \tilde S)$ also
satisfies the conditions (i)--(iii) above, and therefore yields a
quadratic form triple, called {\it full Clifford system} and denoted
$\FC(q,S)$. Its multiplication is given by
\begin{equation}
P(c_1e_1\oplus m\oplus c_2e_2)\, (b_1e_1\oplus n\oplus
b_2e_2)=d_1e_1\oplus p\oplus d_2e_2,  
\label{eq:cli(1)}
\end{equation} 
where
\begin{eqnarray*} 
d_i &= &c_i^2\, \overline{b_i}+c_i\,
q\big(m,S(n)\big)+\overline{b_j}\, q(m)
   \\
p &=& [c_1\overline{b_1}+c_2\overline{b_2}+q\big(m,S(n)\big)]m+
[c_1c_2-q(m)]S(n)
\end{eqnarray*} and
\begin{eqnarray}
 &\{ c_1e_1\oplus m\oplus c_2e_2\,,\, b_1e_1\oplus n\oplus
    b_2e_2\, ,\,
c_1'e_1\oplus m'\oplus c_2'e_2\}   \nonumber \\
& = d_1e_1\oplus p\oplus d_2e_2,  \label{eq:cli(2)}  
\end{eqnarray}
where
\begin{eqnarray*}
d_i &=& q\big(c_im'+c_i'm\,,\,S(n)\big)\, +\, b_j\, q(m,m')
    +2c_i\, c_i'\, \overline{b_i}\\ 
p&=&[c_1\overline{b_1}+c_2\overline{b_2}+q\big(m,S(n)\big)]\, m'+
      [c_1'\overline{b_1}+c_2'\overline{b_2}+q\big(m',S(n)\big)]\, m
       \\ 
&\quad  & +\, [c_1c_2'+c_1'c_2-q(m,m')]S(n).
\end{eqnarray*} 
Note that  $\FC(q,S)$ is graded-triangulated by $(u; e_1,e_2)$. \sm

As already observed in \cite[3.5]{MNSimple}, in general we need not
take the full Peirce spaces $De_i$ in order to get a
graded-triangulated Jordan triple system. Indeed, let us define a
{\it Clifford-ample subspace of $(D,\; \bar{}\;, q)$\/}  as a
 graded $k$-submodule $D_0$ of $D$, such that $D_0=\overline{D_0}$, $1\in D_0$ and
$D_0q(M)\subseteq D_0$. Then
$$M_0 := D_0\,e_1\oplus M\oplus D_0\,e_2$$
is a graded subsystem of the full Clifford system $\FC(q,S)$
containing the triangle $(u; e_1,e_2)$. Hence it is a
graded-triangulated Jordan triple system, called an {\it ample
Clifford system} and denoted $\AC(q,S, D_0)$ or $\AC(q,M,S,
\allowbreak D, \allowbreak {}^-,D_0)$ if more precision is
necessary. Note that $J_0=\AC(q, S,\allowbreak D_0)$ is an outer
ideal of the full Clifford system $J=\FC(q,S)$.

We point out that $(D, \Id, \;\bar{}\;)$ is a  graded associative
coordinate system in the sense of Def.~\ref{examplehermi}, and that
a Clifford-ample subspace $D_0$ is in particular $(\Id,
\;\bar{}\;)$-ample. Hence ample Clifford systems are full in
characteristic $\neq 2$, which here means $D_0=D$.

Our derived operations of $\S\ref{sec:gr-triang JTS general}$ on
$J_0$ are
\begin{eqnarray*}
\overline{d_0e_1\oplus m\oplus c_0e_2} &= & \overline{d_0} e_1\oplus
S(m)\oplus \overline{c_0}e_2, \\ (d_0e_1\oplus m\oplus c_0e_2)^* &=&
c_0e_1\oplus (q(u,m)u-m)\oplus d_0e_2.
\end{eqnarray*}
\end{definition}

\begin{theorem} \label{clifford} {\rm \textsc{Clifford
coordinatization theorem}}. {\rm (\cite[3.6, 3.10]{MNSimple} for
$\La=0$)} Let $J=J_1\oplus M\oplus J_2$ be a  graded Jordan triple
system which is faithfully triangulated by $(u; e_1,e_2)$.  For
$i=1,2$, define
\begin{enumerate}
\item[{\rm (i)}] $C_i$ as the subalgebra of End$_k(M)$ generated by
$L(J_i)$,
\item[{\rm (ii)}] $\Gamma_i (x_i;m):=L\big(T_i(x_i\cdot m)\big)-L\big(T_i(m)\big)L(x_i)\in
C_i$ for $x_i\in J_i, m\in M$,
\item[{\rm (iii)}] $\Delta_i (x_i ;m)=L(P(m)P(u)x_i)-L\big(Q_i(m)\big)L(x_i)\in
C_i$,
\item[{\rm (iv)}] $N_0=\{n_0\in M: \, \Ga_i (J_i; n_0)=0,\,
\,i=1,2\}$,
\item[{\rm (v)}] $K_i=\{k_i\in J_i: \, \Ga_i (k_i; C_iN_0)=0\}$,
\item[{\rm (vi)}] $N^{\rm gr}=\bigoplus_{\la \in \La}(N\cap M^\la)$, for $N=\{n\in
M: \, \De_i (J_i; n)=\De_i (J_i; n,N_0)= \Ga_i (T_i(n); C_iN_0)=0,\,
\,i=1,2\}\subseteq N_0$, i.e., $N^{\rm gr}$ is the greatest
 graded  submodule of $N$.
\end{enumerate}
Then
$$J_q=K_1\oplus N^{\rm gr}\oplus K_2$$
is a graded subsystem of $J$ which is faithfully triangulated by
$(u; e_1,e_2)$ and graded isomorphic to the
 ample Clifford system $\AC(q, N^{\rm gr},
\allowbreak S, D, {}^-,D_0)$ under the map
$$x_1\oplus n\oplus x_2\mapsto L(x_1) \oplus n\oplus L(x_2^*),$$
where $D_0=L(K_1)$, $D$ is the subalgebra of End$_k(N^{\rm gr})$
generated by $D_0$, $\overline{c}=P(e)\circ c \circ P(e)|_{N^{\rm
gr}}$, $q(n)=L\big(Q_1(n)\big)$, and $S(n)=P(e)n$.
 The above isomorphism maps the triangle of $J$ onto
the standard triangle of $\AC(q, N^{\rm gr}, \allowbreak S, D,
{}^-,D_0)$.

Moreover, $J=J_q$ as  graded  triple systems if and only if
$\Delta_1(J_1 ;M)\equiv 0$. In particular, if $u$ is $C_1$-faithful
and $(x_1-x_1^*)\cdot m=0$, for all $x_1\in J_1$, $m\in M$, then
$\Delta_1(J_1 ;M)\equiv 0$ and so $J=J_q$.
\end{theorem}\begin{proof} Let $J$ be a  graded  Jordan triple system
faithfully triangulated by $(u; e_1,e_2)$. By \cite[3.10]{MNSimple},
$K_1\oplus N \oplus K_2$ is a subsystem of $J$ faithfully
triangulated by $(u;e_1,e_2)$. Since $N_0$ and $K_i$, $i=1,2$, are
 graded  (all the defining identities are linear),
 $J_q$ is the greatest  graded  submodule of
$K_1\oplus N \oplus K_2$. Clearly $(u;e_1,e_2)\in J_q^0$. To see
that $J_q$ is also a  graded   subsystem of $J$ faithfully
triangulated by $(u;e_1,e_2)$, it is in view of \ref{sub} enough to
prove $\overline{N^{\rm gr}}=N^{\rm gr}$ and  $K_i\cdot N^{\rm
gr}\subseteq N^{\rm gr}$. But this follows directly {f}rom
$\overline{N}=N$, $K_i\cdot N\subseteq N$, $J_i^\la\cdot M^\ga
\subseteq M^{\la +\ga}$ and the fact that $^-$ is homogeneous of
degree $0$. On the other hand, since $\Delta_1(K_1 ;N^{\rm
gr})\equiv 0$  by definition, we have by \cite[3.6]{MNSimple} that
$J_q$ is isomorphic to the ample Clifford system $\AC(q, N^{\rm gr},
\allowbreak S, D, {}^-,D_0)$ under the map and data specified in the
theorem. Clearly the isomorphism is homogeneous of degree $0$ and
preserves the triangles.

Recall {f}rom \cite[3.10]{MNSimple} that $J=K_1\oplus N \oplus K_2$
if and only if $\Delta_1(J_1 ;M)\equiv 0$. In this case $N=M$ is
 graded   whence $J_q=J$. Conversely, if $J_q=J$ then
$\Delta_1(J_1 ;M^\la)\equiv 0$, which implies that $\Delta_1(J_1
;M)\equiv 0$: $\Delta_1(J_1 ;m^\la + m^\ga)= \Delta_1(J_1 ;m^\la ,
m^\ga)\subseteq \Delta_1(J_1 ;m^\la, N_0)\equiv 0$ since $N\subseteq
N_0$. Finally, if $u$ is $C_1$-faithful and $(x_1-x_1^*)\cdot m= 0$,
then $\Delta_1(J_1 ;M)\equiv 0$ by \cite[3.8]{MNSimple} and so
$J=J_q$.
\end{proof}

\begin{proposition} \label{cliffordsimpl} Let $J=\AC(q,M, S,D, {}^-, D_0)
=D_0e_1\oplus M\oplus D_0e_2$ be an ample Clifford system for which
$D$ acts faithfully on $M$.
\begin{enumerate}
\item[{\rm (i)}] If $J$ is graded-(semi)prime, then $(D, ^-)$ is
graded-(semi)prime. \sm
\item[{\rm (ii)}] The following are equivalent:

\begin{enumerate}
 \item[{\rm (a)}] $J$ is graded-nondegenerate,

 \item[{\rm (b)}] $q$ is graded-nondegenerate and $J$ is
graded-semiprime,

 \item[{\rm (c)}] $q$ is graded-nondegenerate and $(D, ^-)$ is
graded-semiprime.   \end{enumerate} \sm

\item[{\rm (iii)}] $J$ is graded-strongly prime iff $q$ is
graded-nondegenerate and $(D, ^-)$ is graded-prime. \sm

\item[{\rm (iv)}] {\rm (\cite[3.14, 3.15]{MNSimple} for
$\La=0$)} The following are equivalent:
\begin{enumerate}
 \item[{\rm (a)}]  $J$ is graded-simple,

 \item[{\rm (b)}] $q$ is graded-nondegenerate and $(D, ^-)$ is
graded-simple,

 \item[{\rm (c)}] $J$ is graded isomorphic to one of the following:
\begin{enumerate}
\item[{\rm (I)}] $\AC(q, S, F_0)$ for a graded-nondegenerate $q$ over a
graded-field  $F$ with Clifford-ample subspace $F_0$,  or
\item[{\rm (II)}] a polarized $\AC(q, S, F_0)\oplus \AC(q, S, F_0)$, where
$\AC(q, S,F_0)$ is as in {\rm (I)}.
\end{enumerate}
\end{enumerate}
In this case, $u$ is $D$-faithful.
\end{enumerate}
\end{proposition}

\begin{proof}
We will first establish some preliminary results on the structure of
graded ideals of $J$. Thus, let $I$ be a graded ideal of $J$. We
will use the multiplication formulas (\ref{eq:cli(1)}),
(\ref{eq:cli(2)}) to evaluate the possibilities for $I$. First,
invariance of $I$ under $P(e_i)$, $P(e_1,e_2)$ and $P(u)$ shows that
\begin{gather} I=B_0e_1\oplus N\oplus B_0e_2,\tag{$1$}\end{gather}
for some graded $k$-submodules $B_0 \subseteq D_0$, $N\subseteq M$.
{F}rom $P(e_1)I\subseteq I$ we obtain $B_0 =\overline{B_0}$.
Furthermore, we claim
\begin{gather} q  \text{ graded-nondegenerate and }\, B_0=0 \Rightarrow I=0. \tag{$2$}\end{gather}
Indeed, in this case $P(N)e_2+\{N,e_2,M\}\subseteq B_0$ implies
$q(N)+q(N,M)\subseteq B_0=0$, whence $N^\la \subseteq \GR q$ for all
$\la \in \La$, and then $N=0$, therefore $I=0$. Next we claim
\begin{gather} (D, ^-) \text{ graded-simple}, \,  q  \text{ graded-nondegenerate}
\Rightarrow J \, \text{ graded-simple}.\tag{$3$}\end{gather} First
notice that by Cor.~\ref{patrick} either $D$ is a division-graded
algebra or is the direct sum of two copies of a division-graded
algebra with the exchange automorphism. Let $I$ be a proper graded
ideal of $J$ which we write in the form (1). Then $B_0\neq D$, since
$B_0=D$ would imply $e_1\in I$, which in turn would force
$e_2=P(u)e_1\in I$, $e_1+e_2\in I$ and then $I=J$. Now, as in the
proof of \cite[3.14]{MNSimple}, it  follows that no $b_0 \in B_0$ is
invertible in $D$: If $b_0^{-1}\in D$, then
$\overline{b_0^{-2}}=q(\overline{b_0^{-1}}u) \subseteq D_0$, since
$q(M)\subseteq D_0$, and
$e_1=b_0^2\overline{\overline{b_0^{-2}}}e_1=P(b_0e_1)(\overline{b_0^{-2}}e_1)\in
P(I)J\subseteq I$. Hence, if $D$ is division-graded, $B_0^\la=0$ for
all $\la \in \La$, and then $B_0=0$. If, otherwise, $D=A\boxplus A$
for a division-graded $A$ with the exchange automorphism, let
$b_0=(a,0)$ or $(0,a)$ be in $B_0$ for a homogeneous $a\in A^\la$.
Then $(a,a)=b_0+\overline{b_0}\in B_0$, and since $(a,a)$ is not
invertible in $D$, we get $a=0$ and then $B_0=0$, in which case
$I=0$ by (2).\sm
\begin{gather} (D, ^-) \, \text{ graded-(semi)prime and }\, q \,
 \text{ graded-nondegenerate }\Rightarrow  \tag{$4$} \\
 J \, \text{ graded-(semi)prime}.\notag\end{gather} We suppose $I, K$ are
graded ideals of $J$ with $I=K$ in the semiprime case, satisfying
$P(I)K=0$. By (1), we can write $I,K$ in the form $I=B_0e_1\oplus
N\oplus B_0e_2$ and $K=C_0e_1\oplus L\oplus C_0e_2$.
 {F}rom $P(B_0e_1)C_0e_1=0$ we get $B_0^2C_0=0$. If $D$ is
graded-(semi)prime (always in the semiprime case by
Rem.~\ref{semip}), then $B_0=0$ or $C_0=0$. If $D$ is not
graded-prime, then it easily follows that $D$ is a subdirect sum
$A\boxplus_{\rm s}A$ of two copies of a graded-prime algebra $A$
with the exchange automorphism. In this case, let $0\neq b_0\in
B_0$, $0\neq c_0\in C_0$ be homogeneous elements. Then, by
graded-primeness of $A$, $b_0^2c_0=0$ implies that $b_0=(b, 0),
c_0=(0, c)$ or $b_0=(0, b), c_0=(c,0)$ for homogeneous $b, c\in A$,
respectively. Without loss of generality, assume $b_0=(b, 0)$ and
$c_0=(0, c)$. Hence $b_0+\overline{b_0}=(b, b)\in B_0$ and $0=(b^2,
b^2)(0, c)=(0, b^2c)$, thus $b^2c=0$. But, again by the
graded-primeness of $A$ we have that  $b^2=0$ or $c=0$, that is,
 $b=0$ or $c=0$, which is a contradiction. Then $B_0=0$ or
$C_0=0$. Therefore $I=0$ or $K=0$ by (2).

For the proof of the other directions we again establish some
preliminary results. Let $B$ be a graded ideal of $(D, ^-)$. A
straightforward verification using the multiplication rules
(\ref{eq:cli(1)}), (\ref{eq:cli(2)}) shows that then
$$\tilde{B}:=(B\cap D_0)e_1\oplus BM\oplus (B\cap D_0)e_2$$ is a
graded ideal of $J$. Since $BM=0$ implies $B\subseteq {\rm
Ann}_D(M)=0$, it is clear that
$$\tilde{B}=0\Leftrightarrow B=0.$$ On the other hand, if
$\tilde{B}=J$, then $1\in D_0=B\cap D_0$, hence $B=D$. Then
$$\tilde{B}=J\Leftrightarrow B=D.$$
It now follows easily {f}rom the multiplication rules in
Def.~\ref{cli} that
\begin{gather} J\, \text{ graded-(semi)prime }\Rightarrow (D, ^-)\,
   \text{ graded-(semi)prime, and } \tag{$5$}\end{gather}
\begin{gather} J \, \text{ graded-simple }\Rightarrow (D, ^-) \,
\text{ graded-simple}. \tag{$6$}\end{gather}

\noindent For the proof of (ii) we also need
\begin{gather} (D, ^-) \, \text{ graded-semiprime }
\Rightarrow J_1 \, \text{ graded-nondegenerate}.
\tag{$7$}\end{gather} Indeed, if $d_0\in D_0$ is a homogeneous
trivial element of $J_1$, then $d_0^2D_0=0$. In particular,
$d_0^2=0$. But in a graded-semiprime commutative algebra, all
homogeneous nilpotent elements vanish, so $d_0=0$.

Finally, by using Prop.~\ref{peirce}(i) and the fact that $Q_1=q$ in
our situation we have
\begin{gather} J \, \text{ graded-nondegenerate }
\Leftrightarrow \tag{$8$}\\q \, \text{ graded-nondegenerate and }\,
J_1 \, \text{ graded-nondegenerate}.\notag\end{gather}

Now the proof of (i)-(iv) follows easily: (i) is (5). For (ii), the
fact that any graded-nondegenerate Jordan triple system is also
graded-semiprime together with (i), (7) and (8) yields: $J$
graded-nondegenerate (by (8)) $\Rightarrow$ $q$ graded-nondegenerate
and $J$ graded-semiprime (by (i)) $\Rightarrow$ $q$
graded-nondegenerate and $(D, ^-)$ graded-semiprime (by (7))
$\Rightarrow$ $q$ graded-non\-de\-ge\-ne\-rate and $J_1$
graded-nondegenerate (by (8)) $\Rightarrow$ $J$
graded-nondegenerate.

For (iii), we have by definition of graded-strongly primeness and
(i) and (ii) that $J$  graded-strongly  prime implies $q$
graded-nondegenerate and $(D, ^-)$ graded-prime. The converse
direction follows {f}rom (ii) and (4). Finally (iv) follows {f}rom
(3), (6) and graded-nondegeneracy of $J$ characterized by (8). Note
that then $u$ is $D$-faithful since $\{d\in D: du=0\}$ is a proper
graded ideal of $(D, ^-)$. The remaining statements in (iv) are an
immediate application of Cor.~\ref{patrick}.
\end{proof}

The assumption that $D$ acts faithfully on $M$ in the preceding
proposition and the following corollary will be automatic in the
application later on.

\begin{corollary}  \label{clifforddivcor} 
Let $J=\AC(q,S,D_0) = D_0e_1\oplus M\oplus D_0e_2$ be an ample
Clifford system with $D$ acting faithfully on $M$. \sm

{\rm (a)} The following are equivalent:
\begin{enumerate}
 \item[\rm (i)] Every nonzero homogeneous element of $M$ is invertible,

 \item[\rm (ii)]  $D$ is a graded-field and $q$ is {\rm graded-anisotropic} in
the sense that $0\ne q(m)\in D$ for every nonzero homogeneous $m\in
M$,
\item[\rm (iii)]  $J$ is division-triangulated.
\end{enumerate}
In this case $M$ is a free $D$-module. \ms

{\rm (b)} Let $k$ be a field. Then $J$ is a triangulated Jordan
triple torus iff

\begin{enumerate}

 \item[{\rm (I)}] $D$ is a torus, say with $\supp_\La D = \Ga$,
hence $D=k^t[\Ga]$ is a twisted group algebra, and

 \item[{\rm (II)}] $M$ is a free $D$-module with a homogeneous
$D$-basis $\{u_i: i\in I\}$, say $u_i \in M^{\de_i}$, with $q(u_i)
\ne 0$ and $(\de_i + \Ga) \ne (\de_j + \Ga) $ for $i\ne j$.
\end{enumerate}
If in this case $D=D_0$, then $\{u_i : i \in I\}$ is an orthogonal
basis: $q(u_i, u_j) = 0$ for $i\ne j$.
\end{corollary}

\begin{proof} (a) If (i) holds, the invertibility criterion
(\ref{eq:qfinvert}) together with $q(du) = d^2$ implies that $D$ is
a graded-field and then that $q$ is graded-anisotropic. Suppose
(ii). Then clearly every nonzero homogenous $m\in M$ is invertible.
Moreover, it follows as in Lem~\ref{hermtor} that $D_0$ is a
division-graded triple, whence $J$ is division-triangulated. The
implication (iii) $\Rightarrow$ (i) is clear. It is a standard fact
that any graded module over a division-graded algebra is free with a
homogeneous basis.

(b) Suppose $J$ is a triangulated Jordan triple torus. Then (a)
applies. Since $D \to D.u \subset M$ is injective and homogeneous of
degree $0$, $D$ is a torus. Hence $D=k^t[\Ga]$ is a twisted group
algebra for $\Ga= \supp_\La D$, a subgroup of $\La$. As in (a), $M$
is a free $D$-module with a homogeneous basis $\{u_i :i\in I\}$. We
have $q(u_i) \ne 0$ because $u_i \ne 0$. Since $0 \ne D^\ga M^\mu$
for $\ga\in \Ga$ and $\mu \in \supp_\La M$, the condition $(\de_i +
\Ga) \ne (\de_j + \Ga)$ for $i\ne j$ follows from $\dim_k M^\la\le
1$. The converse is easily verified. Observe that $q(u_i, u_j) \in
D^{\de_i + \de_j}$, but $\de_i + \de_j \not\in \Ga$ if $D=D_0$.
Otherwise, $\de_i = -\de_j + \ga = \de_j + (\ga - 2\de_j)$ for some
$\ga\in \Ga$ and $\ga - 2\de_j\in \Ga$ by (\ref{eq:supplem1}) since
$\Ga = \scL$ in the notation of loc.~cit. and $\scS=-\scS$.
\end{proof}

\begin{example} \label{examaczn} {\rm Let $\La=\ZZ^n$ and let $J=\AC(q,S,D_0)= D_0 e_1\oplus M \oplus
D_0 e_2$ be a $\La$-triangulated ample Clifford system such that
$D_0$ generates $D$ as algebra. For $\scL = \supp_\La D_0$ we
therefore have $\ZZ[\scL] =\supp_\La D = \Ga$, a subgroup of $\La$,
and for $\scS=\supp_\La M$ we get from (\ref{eq:supplem1}) that
$2\scS \subset \scL \subset \scS$, so $\La= \ZZ[\scS]$ and $2\La
\subset \ZZ[\scL]=\Ga \subset \La$, proving that $\La / \Ga$ is a
finite group. It follows that $\Ga$ is free of rank $n$, thus $D$ is
isomorphic to a Laurent polynomial ring in $n$ variables. Moreover,
from $\scL + \scS \subset \scS$ (or from
Cor.~\ref{clifforddivcor}(b)) we get $\Ga + \scS \subset \scS$,
whence the set of coset $\scS/\Ga$ embeds in the finite group
$\La/\Ga$ and is therefore finite. Thus, $M$ is free of finite rank.
}\end{example}

\section{Graded-simple-triangulated Jordan triple systems}
\label{sec:grsimcla}

In this section we prove (Th.~\ref{maintheorem}) that under some
mild additional assumptions the graded-simple hermitian matrix and
ample Clifford systems give us in fact all the possibilities for
graded-simple-triangulated Jordan triple systems, and we describe
them completely in Cor.~\ref{classification}.  Finally, we describe
the division-triangulated and tori, in particular the case
$\La=\ZZ^n$, among the graded-triangulated Jordan triple systems in
Corollaries~\ref{divgradclassi}, ~\ref{jtstorclassi} and
~\ref{jtszncla}.

Unless specified otherwise, $J=J_1 \op M \op J_2$ is a Jordan triple
system over $k$ triangulated by $(u; e_1, e_2)$. We refer the reader
to \S\ref{sec:gr-triang JTS general} for unexplained notation. We
will not right away assume that $J$ is graded  or even
graded-simple. Rather, to prove the main result of this section we
will perform certain reductions to more specific situations (passing
to a completion of $J$ over the Laurent series ring or passing to an
isotope) and, unfortunately, graded-simplicity can not always be
maintained under these reductions. We will therefore begin this
section by presenting these reductions. \sm

Let $J$ be an arbitrary Jordan triple system and let $t$ be an
indeterminate over $k$. We denote by
$$ \hj = J((t)) = \big\{ \ts \sum_{i\ge N} \, x_i t^i : x_i \in J, N\in
\ZZ \big \}
$$
the Jordan triple system over $k$ whose Jordan triple product is
defined by
\begin{eqnarray*}&& \wh P\big(\sum_{i\ge N}\, x_i t^i\big) \,
  \big(\sum_{j\ge M} y_j t^j\big)\\ &&  = \sum_{i\ge N, \, j\ge M}\, P(x_i)y_j t^{2i+j} \, + \,
\sum_{i_2> i_1\ge N, \, j\ge M}\, \{x_{i_1},\, y_j,\, x_{i_2} \}
t^{i_1 + i_2 + j}. \end{eqnarray*}

Note that this makes sense since in any fixed degree the sum on the
right hand side is finite. Observe that $\hj$ contains $J=J\, t^0$
as a subsystem. It is also easy to check that $\sum_{i\ge N} \, x_i
t^i$ with $x_N \ne 0$ is invertible in $\hj$ if  $x_N$ is invertible
in $J$.

\begin{assum} \label{ass1}  $J=J_1\oplus M \oplus J_2$
is a Jordan triple system triangulated by $(u; e_1,e_2)$ for which
the $k$-linear map $L: J_1 \to C_0 : x_1 \mapsto L(x_1)$ defined in
{\rm \ref{001}} is injective. {\rm Note that then $L: J_2 \to C_0^*$
is also injective because $L(x_2)c_0^* = (L(x_2^*)c_0)^*$ and $x_2^*
\in J_1$. Recall that $C$ denotes the subalgebra of $\End_k(M)$
generated by $C_0$.}
\end{assum}

\begin{lemma} \label{laurent} If $J$ satisfies {\rm Assumption~\ref{ass1}}
with respect to the triangle $\scT=(u; e_1, e_2)$ then so does
$\hj$, also with respect to $\scT$. Moreover:
\begin{enumerate}
 \item[{\rm (i)}] The Peirce spaces of $\hj$ with respect to $\scT$ are $\hj_i =
J_i((t))$ and $\wh M = M((t))$.

\item[{\rm (ii)}] Let $\wh C$ be the subalgebra of End$_kM((t))$ generated by
$\wh{C_0}= L(\wh J_1)=C_0((t))$, and let $\hat \pi$ be the reversal
involution of $\wh C$ with respect to $\wh{C_0}$ {\rm (cf.
\ref{pi})}. Then $C$ is canonically isomorphic to the subalgebra of
$\wh C$ preserving degrees. Identifying $C$ with this subalgebra we
have $\wh \pi _{| C} = \pi$ and, with obvious meaning, $\wh * _{| C}
= *$.

\item[{\rm (iii)}] For $m\in M$ we put $\wh m = u + t m$ and note that $\wh
m$ is invertible in $J((t))$. Then for all $c\in C\subseteq \wh C$
and $\wh Q_2(.)=\wh P(.)e_1 :$\sm
\begin{eqnarray*}
 {\wh Q_2}(c\wh m)=0= {\wh Q_2}(c \wh m, \wh m ) &\implies &
           Q_2(cm)=0=Q_2(cm, m), \\
 {\wh Q}_2 (\wh{m},c\, c^*\wh{m})=0 &\implies&
              Q_2(m,cc^*m)=0, \hbox{ and} \\
 {\wh Q}_2 (cM, \wh m )=0 &\implies&  Q_2(cM,m)=0.
\end{eqnarray*}

\item[{\rm (iv)}] If $J_i$ does not contain nonzero elements with trivial
square {\rm cf. (\ref{eq:squares})}, then neither does $\hj_i$.
\end{enumerate}
\end{lemma}

\begin{proof} (i) and (ii) are clear. (iii) That $\wh m =u+tm$ is
invertible in $\hj$ follows {f}rom the invertibility criterion
mentioned above.  We have $ {\wh Q}_2 (c \wh m) = {\wh P}
(cu+ctm)e_1 = Q_2(cu) + Q_2(cu,cm)t + Q_2(cm)t^2$ and $ {\wh Q}_2
(c\wh m , \wh m )=
Q_2(cu,u)+\big(Q_2(cu,m)+Q_2(cm,u)\big)t+Q_2(cm,m)t^2$, which
implies the first equation. The others follow similarly. \end{proof}

Our second reduction is passing to an isotope. Recall that for an
arbitrary Jordan triple system and an invertible element $v$ of $J$
the isotope $J^{(v)}$ is the Jordan triple system with
multiplication $P^{(v)}(x)y=P(x)P(v)y$. The following lemma, whose
proof is left to the reader, describes which properties are
maintained by passing {f}rom $J$ to a special isotope.

\begin{lemma} \label{isotope} Suppose $J$ is triangulated by $\scT=(u; e_1,
e_2)$, and let $m\in M$ be an invertible element. Then
$v=e_1+Q_2(m)^{-1}$ is invertible in $J$ and the isotope $\wt J
:=J^{(v)}$ with $\wt P=P^{(v)}$ is triangulated by $\wt \scT= (\wt u
; \wt e_1,\wt e_2)=(m;e_1,Q_2(m))$ with Peirce spaces $\wt J_1
=J_1$, $\wt M =M$ and $\wt J_2 =J_2$ as $k$-modules. Moreover,
denoting the data for $\wt J$ by $\wt L, \wt C_0$ etc, we have :
\begin{enumerate}
 \item[{\rm (i)}] $\wt L=L$ as $k$-linear maps, hence $\wt C_0=C_0$ and $(\wt C, \wt
\pi)=(C,\pi)$ as algebras with involution. In particular, if $J$
satisfies {\rm Assumption~\ref{ass1}} then so does $\wt J$ with
respect to $\wt \scT$.

 \item[{\rm (ii)}] For  $n, n_1 \in M$ we have ${\wt Q_2}(n)=Q_2(n)$ and
${\wt Q_2}(n,n_1) = Q_2(n,n_1)$.

 \item[{\rm (iii)}] If $J_1$ does not contain nonzero $x_1\in J_1$ with
$x_1^2=0$ the same holds for $\wt J_1$.

 \item[{\rm (iv)}] Suppose $J= \bigoplus_{\la \in \La} J^\la$ is
graded-triangulated by $(u; e_1,e_2)\in J^0$. If\/  $m\in M^\la$ is
homogeneous, then $\wt J$ is graded-triangulated  with
$$\wt J_1^\mu=J_1^\mu , \quad \wt M^\mu=M^{\mu+\la} ,\quad
\wt J_2^\mu=J_2^{\mu+2\la} \quad  (\mu \in \La). $$ \end{enumerate}
\end{lemma}

We point out that $\wt J_1 $ and $J_1$ are in general not isomorphic
as triple systems, rather we have $\wt P(x_1)y_1 = P(x_1){\bar
y_1}$. \ms

{F}rom now on we will use the notations $\wh J$ and $\wt J$ to
denote the Jordan triple systems of Lemma~\ref{laurent} and
Lemma~\ref{isotope}. \ms

\begin{proposition} \label{tec} Suppose $J$
satisfies {\rm Assumption~{\rm \ref{ass1}}}, and let $R$ be a
$\pi$-invariant ideal of $C$ satisfying $R\cap C_0=0$. Then for all
$r\in R$, $x_1\in J_1$, $c \in C $ and $m\in M$ the following hold:
\begin{enumerate}
\item[{\rm (i)}]$r+r^\pi = r^2 = rr^\pi = rL(x_1)r^\pi = r(c^\pi
-c)=0$. Also $[r, C]=0$, so $R$ is a central ideal.

\item[{\rm (ii)}] $Q_i(ru)=0=T_i(ru)$ for $i=1,2$,

\item[{\rm (iii)}] $Q_2(rm)=0=Q_2(rm, m)$,

\item[{\rm (iv)}] $Q_1(rm)^2=0$,

\item[{\rm (v)}] $T_2(rm)^2=Q_1(rm)^*$,

\item[{\rm (vi)}] If either {\rm (a)} $M=Cu$ or {\rm (b)} $J_1$  does not contain nonzero $x_1\in J_1$
with $x_1^2=0$ {\rm (cf. \ref{eq:squares})}, then $Q_2(RM,M)=0$.
\end{enumerate}
\end{proposition}

\begin{proof} (i) By \ref{2} and \ref{3},
$r+r^\pi=L\big(T_1(ru)\big)$, $rL(x_1)r^\pi=L(P(ru)P(u)x_1)$ and
$rr^\pi=L\big(Q_1(ru)\big)$ all lie in $R\cap C_0$. Hence
$r+r^\pi=rr^\pi=rL(x_1)r^\pi=0$ and, because of injectivity of $L$,
also  $Q_1(ru)=0=T_1(ru)$. It now follows that $r^2=-rr^\pi=0$.
Linearizing $cL(x_1)c^\pi\in C_0$, we have
\begin{gather} cL(x_1)d^\pi +dL(x_1)c^\pi\in
C_0. \tag{$1$}\end{gather} Specializing (1) for $d=r$ and using that
$r^\pi=-r$, we get $cL(x_1)r =rL(x_1)c^\pi$. For $c=1$ we then have
$L(x_1)r=rL(x_1)$. Since $C_0$ generates $C$ as a $k$-algebra, this
forces $[r, C]=0$. Then $cL(x_1)r =rL(x_1)c^\pi$ evaluated for
$x_1=e_1$  shows $r(c^\pi -c)=0$.

(ii) We have already shown in the proof of (i) that
$Q_1(ru)=0=T_1(ru)=Q_1(ru, u)$. By \ref{6}, $(ru)^* =r^\pi u =-ru$,
and then by \ref{4}, $0=T_1(ru)^* =T_2\big((ru)^*\big)=-T_2(ru)$ and
$0=Q_1(ru)^* =Q_2\big((ru)^*\big)=Q_2(-ru)=Q_2(ru)$.

(iii) We first prove that it is enough to show (iii) for invertible
$m$ by passing to $\wh J$. Indeed, because of (iii) of
Lem.~\ref{laurent} it is enough to establish (iii) for $\wh J$. But
for $\wh J$ we know that for any $m\in M$  the element $\wh m =u+tm$
is invertible in $\wh J$ and that $\wh J$ also satisfies Assumption
\ref{ass1}. Let $\wh R$ be the ideal of $\wh C$ generated by
$R\subseteq C \subseteq \wh C$, that is $\wh R=R((t))$. Then $\wh R
\cap \wh C_0=(R \cap C_0)((t))=0$ follows. Thus, without loss of
generality we can assume that $m$ is invertible.

We then pass to the isotope $\wt J$, and note that $\wt J$ satisfies
the assumptions of this proposition. Moreover, because  of (ii) of
Lem.~\ref{isotope}, it will be sufficient to prove (iii) for $m=\wt
u \in \wt J$, or equivalently for $u \in J$. But (iii) for $m=u$ is
just (ii).

(iv) now follows easily {f}rom (iii) since $Q_2(rm)=0$ implies
$Q_1(rm)^2=P\big(Q_1(rm)\big)e_1=P(rm)P(e_2)P(rm)e_1=P(rm)P(e_2)Q_2(rm)=0$.

(v) We have $T_2(rm)^2=P\big(T_2(rm)\big)e_2=P(\{rm, e_1, u\})e_2$,
where, by \cite[JP21]{JP} and because of $P(e_1)u=0$ and $Q_2(rm)=0$
by (iii),
\begin{eqnarray*}
    P(\{rm, e_1, u\})e_2
      &= & P(rm)P(e_1)P(u)e_2 + P(u)P(e_1)P(rm)e_2 \\
      && + L(rm,e_1)P(u)L(e_1, rm)e_2 -
      P(P(rm)P(e_1)u,u)e_2 \\
      &=& Q_2(rm)+Q_1(rm)^*+Q_2(rm, P(u)\,\overline{rm})
      \\
      &=&Q_1(rm)^* +Q_2(rm, (rm)^*).
\end{eqnarray*}  By \ref{5}, $(rm)^* = r^* m^*
= r^*( T_1(m)\cdot u - m) = L\big(T_1(m)\big)\, r^* u - r^* m $
where in the last equality we used \ref{05} and the fact that $C^*$
is the subalgebra generated by $L(J_2)$, and hence commutes with
$C$. Since $r^\pi = - r$ we then get {f}rom   \ref{pi2} and \ref{6}
that
\begin{eqnarray*}
   Q_2(rm, (rm)^*)
    &= & Q_2\big(rm, r^* (T_1(m)\cdot u-m)\big)\\
    &=& Q_2(m, r^\pi L\big(T_1(m)\big)r^* u) \, - \, Q_2(m,r^\pi r^* m)
    \\
    &=& -Q_2(m, r L\big(T_1(m)\big)r^\pi u)+Q_2(m,r r^* m)\\
    &=&Q_2(m,r r^* m)
\end{eqnarray*} since $r L\big(T_1(m)\big)r^\pi=0$ by (i). Therefore, if we can establish
$Q_2(m,r r^* m)=0$ we are done. As in the proof of (iii) we imbed
$J$ into $\wh J$. Then Lem.\ref{laurent}(iii) shows that it is
sufficient to prove this for an invertible $m$. But for invertible
$m$ we have $C^* m \subseteq Cm$, since by \cite[JP21]{JP} and
\ref{04}
\begin{eqnarray*} L(x_2)m
 &= & P(x_2, m)e_2=P(P(m)P(m)^{-1}x_2, m)e_2\\
 & =& L(m, P(m)^{-1}x_2)P(m)e_2=\{Q_1(m), P(m)^{-1}x_2, m\} \\
& =& L\big(Q_1(m)\big)L(\, \overline{P(m)^{-1}x_2}\, )m \in
Cm.\end{eqnarray*} Hence $Q_2(m, R C^* m)\subseteq Q_2(m, R C
m)\subseteq Q_2(m, R m)=0$ by (iii).

(vi) In case (a) we have, using \ref{3.5}, $Q_2(RM,M)=
Q_2(RCu,Cu)\subseteq Q_2(Ru,Cu) = T_2(R^\pi Cu)= T_2(RCu)\subseteq
T_2(Ru)$. Now the claim follows {f}rom (ii).

In case (b) first note that neither $J_2=J_1^*$ contains a nonzero
$x_2\in J_2$ with $x_2^2=0$. It then follows {f}rom (iv) and (v)
that
\begin{gather} T_2(RM)=0. \tag{$2$}\end{gather}
Next we show
\begin{gather} Q_2(RM, m)=0 \,\text{ for invertible }\, m\in M.
\tag{$3$}\end{gather} Indeed, applying Lem.~\ref{isotope}, in
particular (ii) and (iii), we see that it suffices to prove (3) for
$m=u$, in which case it reduces to (2). Finally, we can show that
$Q_2(RM, m)=0$ for arbitrary $m\in M$: By Lem.~\ref{laurent} it
suffices to show ${\wh Q_2}(RM, \wh m) = 0$, for $\wh m =u+tm$. But
this holds by (3) since $\wh J$ satisfies our assumptions.
\end{proof}

Recall that $C_0$ is a Jordan triple system with
$P(c_0)(d_0)=c_0\overline{d_0}c_0$ and that $L: J_1 \to C_0$,
$x_1\mapsto L(x_1)$ is a nonzero specialization (see \ref{1}). In
particular, $L(x_1^2) = \big(L(x_1)\big)^2$ and if
Assumption~\ref{ass1} holds then $J_1 \cong_\La C_0$ as  graded
Jordan triple systems  if $J$ is graded by $\La$.

\begin{lemma} \label{fund} Suppose $J$ is graded-triangulated and
satisfies {\rm Assumption~\ref{ass1}} with respect to  $(u; e_1,
e_2) \subseteq J^0$. Let $R$ be a $\pi$-invariant graded ideal of
$C$ such that $R\cap C_0=0$. If\/  $C_0 $ does not contain nonzero
homogeneous elements which square to zero, then $Q_2(RM,m)=0$ for
invertible homogeneous $m\in M$.
\end{lemma}

\begin{proof}
We pass {f}rom $J$ to the isotope $\wt J$ of Lem.~\ref{isotope}.
Since by that lemma all our assumptions are maintained, it follows
{f}rom (ii) of Lem.~\ref{isotope} that it suffices to prove
$T_2(RM)=0$. To do so, let $r\in R$ and $n\in M$ be homogeneous
elements. By Prop.~\ref{tec}(iv),
$L\big(Q_1(rn)\big)^2=L(Q_1(rn)^2)=0$, which implies that
$Q_1(rn)=0$ by our assumptions. But then by Prop.~\ref{tec}(v),
$L\big(T_2(rn)\big)^2=L(T_2(rn)^2)=L(Q_1(rn)^*)=0$. Since $C_0
\cong_\La J_1$ does not contain nonzero elements with square $0$,
the same holds for $C_0^* = L(J_2)$. Hence $T_2(rn)=0$ follows, and
this implies $T_2(RM)=0$ as desired.
\end{proof}

\begin{lemma}\label{Cbreve} Suppose $J$ is graded-triangulated by
$(u; e_1, e_2)$. Then $C':=C\{c-c^\pi : c\in C\}C=C[C,C]C$ is a
$(\pi , ^-)$-graded ideal of $C$ such that $C'M\subseteq Cu$.
\end{lemma}

\begin{proof} The first part of the lemma is straightforward. That $C'M \subseteq Cu$ follows {f}rom
\cite[1.6.13]{MNSimple}.
\end{proof}

\begin{assum} \label{ass2} $J$ is graded-triangulated with grading group
$\La$  and fulfills {\rm Assumption~\ref{ass1}} with respect to a
triangle $(u; e_1, e_2) \subseteq J^0$.
\end{assum}

\begin{lemma} \label{lemma_Cbreve} Suppose $J$ fulfills {\rm Assumption~\ref{ass2}}. If $R$ is a maximal $(\pi , {}^-)$-invariant and graded
ideal of $C$ with $R\cap C_0=0$, then $M=Cu$ or $C_0$ does not
contain nonzero homogeneous elements with trivial square.\end{lemma}

\begin{proof} If $R$ is a maximal $(\pi , {}^-)$-invariant and graded
ideal of $C$ with $R\cap C_0=0$, consider the $(\pi ,
{}^-)$-graded-simple algebra
 $\breve{C}:=C/R$ and let $\varphi : C\rightarrow \breve{C}$ be the
canonical epimorphism. Since $\varphi$ is homogeneous of degree $0$,
the ideal
$\varphi(C')=\breve{C}'=\breve{C}[\breve{C},\breve{C}]\breve{C}$ is
 graded   and invariant under the induced maps $\breve{\pi}$ and
$c\mapsto \breve{\overline{c}}$, whence either
$\breve{C}'=\breve{C}$ or $\breve{C}'=0$. If $\breve{C}'=\breve{C}$,
then $\breve{1}\in \breve{C}'$ and so $1=c'+r$ where $c'\in C'$. Now
$1^2=1=c'^2+c'r+rc'+r^2$, but $r^2=0$ by Prop.~\ref{tec}(i). Hence
$1\in C'$ which by Lem.~\ref{Cbreve} implies $M\subseteq Cu$, so
$M=Cu$.

Otherwise $\breve{C}'=0$. Then $[\breve{C},\breve{C}]=0$, that is,
$\breve{C}$ is commutative and hence $\breve{\pi}$ is trivial. In
this case $(\breve{C},\breve{^-})$ is graded-simple. By
Cor.~\ref{patrick}, $\breve{C}$ is either division-graded or the
direct sum of two copies of a division-graded algebra with the
exchange automorphism. In particular,  $\breve{C}$ does not contain
nonzero homogeneous elements with trivial square. So neither does
the subspace $\breve{C}_0=\varphi(C_0)$ nor $C_0$ since $R\cap
C_0=0$. \end{proof}

\begin{proposition} \label{main}
Suppose $J$ is a graded-triangulated Jordan triple system for which
$J_1$ is graded-simple. Then {\rm Assumption~{\rm \ref{ass2}}}
holds. If $R$ is a proper $(\pi,{}^-)$-invariant graded ideal of $C$
then $R\cap C_0=0$ and
 $Q_2(RM,m)\allowbreak =0$ for  invertible homogeneous $m\in
M$. \end{proposition}

\begin{proof} Since $J_1$ is graded-simple and the specialization
$L: J_1 \to C_0, \,x_1\mapsto L(x_1)$ is nontrivial, it is
injective, proving Assumption~\ref{ass2}. It is then an isomorphism
of graded Jordan triple systems, whence $C_0$ is also graded-simple.
Therefore the graded ideal $R\cap C_0$ of the Jordan triple system
$C_0$ must be either $0$ or $C_0$. But if $R\cap C_0=C_0$, then
$C_0\subseteq R$ which implies $C=R$ contradicting that $R$ is
proper. So $R\cap C_0=0$. By Lem.~\ref{lemma_Cbreve}, $M=Cu$ or
$C_0$ does not contain nonzero homogeneous elements of trivial
square. If $M=Cu$, then $Q_2(RM,M)=0$ by Prop.~\ref{tec}(vi).
Otherwise, it follows {f}rom Lem.~\ref{fund} that $Q_2(RM,m)=0$ for
invertible homogeneous $m\in M$. \end{proof}

Now we are ready to establish our main result.

\begin{theorem} \label{maintheorem} {\rm (\cite[4.3]{MNSimple}
for $\La=0$) } Let $J$ be a graded-simple-triangulated Jordan triple
system satisfying one of the following conditions
\begin{enumerate}
\item[{\rm (a)}] every nonzero $m\in M$ is a linear combination of invertible
homogeneous  elements, or
\item[{\rm (b)}] the grading group $\La$ is torsion-free. \end{enumerate}

\noindent Then $(C,\pi, {}^-)$ is graded-simple, $u$ is
$C$-faithful, and exactly one of the following two cases holds:
\begin{enumerate}
\item[{\rm (i)}]  $C$ is not commutative and $M=Cu$. In this case, $\pi\ne \Id$ and $J$ is graded isomorphic to
the graded-simple diagonal hermitian matrix system
 $\rmH_2(C,C_0, \pi , ^-)$;

\item[{\rm (ii)}] $C$ is commutative. In this case, $\pi=\Id$ and
$J$ is graded isomorphic to the graded-simple ample Clifford system
$\AC(q, M, S, C,\allowbreak {}^-, C_0)$, where
$q(m)=L\big(Q_1(m)\big)$ and $S(m)=\overline m$.
\end{enumerate}
In both cases, the triangles are preserved by the isomorphisms.
\end{theorem}

{\bf Remarks.} (1) We point out that the assumptions (a) or (b) are
only needed to show that $(C, \pi, {}^-)$ is graded-simple. Our
proof shows that {\it any graded-triangulated Jordan triple system
with a graded-simple $(C,\pi,{}^-)$ satisfies} (i) {\it or} (ii)! We
also note that hermitian matrix systems are of course also defined
for commutative coordinate algebras $C$. But in the commutative case
they are isomorphic to ample Clifford systems. \sm

(2) This theorem generalizes \cite[Prop.~4.3]{MNSimple}: $\La=0$ is
a special case of our assumption (b). Our proof is slightly
different {f}rom the proof given in \cite{MNSimple} and in fact
corrects a small inaccuracy there: The Isotope Trick
\cite[4.1]{MNSimple} cannot be applied since $\wt J$ does not
necessarily inherit simplicity {f}rom $J$.

(3) The assumption (a), which, admittedly, looks somewhat funny at
first sight, is fulfilled in the most important application of the
theorem, the division-triangulated case (Cor.~\ref{divgradclassi}).

\begin{proof} We will proceed in four steps.

(I) {\it $(C,\pi, {}^-)$ is graded-simple.} Let $R$ be a maximal
$(\pi , {}^-)$-invariant graded ideal of $C$. Such an ideal $R$
exists by Zorn. Our claim (I) then means $R=0$. This will follow if
we can show $rm=0$ for homogeneous $r\in R$ and $m\in M$. Recall
that $Q_2$ is graded-nondegenerate by Prop.~\ref{peirce}(ii). It is
therefore sufficient to prove $$ Q_2(rm)=0=Q_2(rm, M) \quad\hbox{for
homogeneous $r\in R$ and $m\in M$}. \eqno{(1)}$$ Since $J_1$ is
graded-simple by Prop.~\ref{peirce}(ii), it follows {f}rom
Prop.~\ref{main} that $J$ satisfies  Assumption \ref{ass2}. Since
$R$ is in particular proper, it also follows {f}rom Prop.~\ref{main}
that $C_0\cap R=0$ and $Q_2(rm, n)=0$ for invertible homogeneous
$n\in M$. Also, $Q_2(rm)=0$ by Prop.~\ref{tec}(iii). Thus, (1) holds
in case (a).

We also know {f}rom Lem.~\ref{lemma_Cbreve} that $M=Cu$ or $C_0$
does not contain nonzero homogeneous elements with trivial square.
But if $M=Cu$ then $Q_2(rm,M)=0$ by Prop.~\ref{tec}(vi), hence again
(1) follows. We can therefore assume that $C_0$ does not have
nonzero homogeneous elements with trivial square. We will use our
assumption (b) to prove (1) in this case. We claim that in fact
$C_0$ does not contain nonzero elements with trivial square: Let
$x=\sum x^{\la_i}\in C_0$, with $0\neq x^{\la_i}\in C_0^{\la_i}$,
such that $x^2=0$. Since $\La$ is torsion-free, it can be ordered
(as a group) and we can therefore consider $(x^\la)^2$ for $\la
={\rm max}\{\la_i
 \}$. But $x^2=0$ implies $(x^\la)^2=0$, hence
$x^\la=0$ by the absence of nonzero homogeneous elements of trivial
square, contradiction. Since $(J_1, e_1) \cong_\La (C_0, 1)$ as
triple systems with tripotents, the subspace $J_1$ does not contain
nonzero elements with trivial square. But then $Q_2(rm,M)=0=Q_2(rm)$
follows {f}rom (vi) and (iii) of Prop.~\ref{tec}.

(II) {\it $u$ is $C$-faithful.} By (I), we have that $C$ is $(\pi ,
^-)$-graded-simple. Now, $Z=\{z\in C:\, zu=0\}$ is obviously a left
ideal of $C$. It is also a right ideal since for $d\in C$ and $z\in
Z$ we have $zCu=zC^\pi u$ (by \ref{6}) $=zC^* u$ (by \ref{05}) $=C^*
zu=0$. Also, $Z$ is  graded  since $u\in M^0$, and finally it is
$(\pi , ^-)$-invariant since $z^\pi u=(zu)^*=0$ by \ref{6} and
$\overline{z}u=\overline{z} \,
 \overline{u} =\overline{zu}=0$. Then $Z$ must be $C$ or $0$.
But note that $Z\neq C$ since $1\notin Z$. Hence $Z=0$, that is, $u$
is $C$-faithful.

We will now distinguish the two cases $\pi \neq \text{Id}$ and $\pi
= \text{Id}$.

(III) {$\pi \neq \text{Id}$}:  Then $C$ is noncommutative. Indeed,
since $C_0\subset H(C,\pi)$ generates $C$ as an algebra, $C$ is
commutative iff $\pi = \Id$.  Also, there exists a homogeneous $c\in
C$ such that $c^\pi \neq c$, and then the $(\pi , ^-)$-graded ideal
$C'=C\{c-c^\pi : c\in C\}C$ (Lem.~\ref{Cbreve}) is nonzero and hence
equals $C$, in particular $1 \in C'$. By Lem.~\ref{Cbreve} again,
this implies $M=Cu$. By Th.~\ref{hermitian} $J$ is then graded
isomorphic to the
 hermitian matrix system $\rmH_2(C,C_0, \pi , ^-)$
as claimed in (i).

(IV) {$\pi = \text{Id}$:} Then $C$ is commutative and
${}^-$-graded-simple. We will prove that in this case $J$ is graded
isomorphic to an ample Clifford system. By $C$-faithfulness, this
will follow {f}rom Th.~\ref{clifford} as soon as we have established
$(x_1-x_1^*)\cdot m=0$ for all $x_1\in J_1$ and $m\in M$. Now by
\ref{7} we know $(x_1-x_1^*)\cdot m =\Gamma_1(x_1;m)u$. By linearity
of $\Ga_1$ in $x_1$ and $m$, it therefore remains to prove
\begin{gather} \Gamma_1(x_1;m)=0  \hbox{ for homogeneous
$x_1\in J_1$ and $m\in M$. } \tag{$2$}\end{gather} 
Since $C$ is commutative, \ref{8} shows $\Gamma_1(x_1;m)^2m=0$, so
$\Gamma_1(x_1;m)$ is never invertible. If $(C,^-)$ is graded-simple,
it is a division-graded algebra and so (2) holds. Otherwise, by
Cor.~\ref{patrick}, identify $C=A\boxplus A$ with the direct sum of
two copies of a division-graded commutative algebra $A$ and ${}^-$
is the exchange automorphism. Then we have  that
\begin{gather}
\Gamma_1(x_1;m)=0 \hbox{ for all homogeneous ${}^-$-invariant $x_1
\in J_1$ and $m\in M$}. \tag{$3$}\end{gather} Let $x_1\in J_1$ and
$m\in M$ be arbitrary homogeneous elements. Since $C=A \boxplus A$
we get $1=\epsilon +\overline{\epsilon}$ for orthogonal idempotents
$\epsilon$ and $\overline{\epsilon}$ in $C^0$, namely $\epsilon =
1_A$. 
We now claim that
$$(x_1-x_1^*)\cdot m=\epsilon (y_1-y_1^*)\cdot
n_1+\overline{\epsilon}(z_1-z_1^*)\cdot n_2 \eqno{(4)}$$ for some
homogeneous $y_1=\overline{y_1}$ and  $z_1=\overline{z_1}$ in $J_1$
and $n_i=\overline{n_i}\in M$. The proof of (4) given in the proof
of \cite[4.4]{MNSimple} for $\La=0$ also works in our setting. But
(4) together with (3) implies (2), finishing the proof of the
theorem. \end{proof}

From the previous Theorem \ref{maintheorem} together with
Prop.~\ref{hermconclusion} and Prop.~\ref{cliffordsimpl}(iv) we get
the following classification.

\begin{corollary}\label{classification} {\rm  (\cite[4.4]{MNSimple} for
$\La=0$)} A graded-simple-triangulated Jordan triple system
satisfying {\rm (a)} or {\rm (b)} of\/ {\rm Th.} \ref{maintheorem}
is graded isomorphic to one of the following triple systems:

\noindent  non-polarized
\begin{enumerate}
\item[{\rm (I)}] diagonal $\rmH_2(A,A_0, \pi , ^-)$ for a graded-simple
noncommutative $A$ with graded involution $\pi$ and automorphism
${}^-$;

\item[{\rm (II)}] $\Mat_2(B)$ with $P(x)y=x\overline{y}x$
for a noncommutative graded-simple associative unital $B$ with
graded automorphism $^-$ and
$\overline{(y_{ij})}=(\overline{y_{ij}})$ for $(y_{ij})\in
\Mat_2(B)$;

\item[{\rm (III)}] $\Mat_2(B)$ with $P(x)y=x\overline{y}^tx$
for a noncommutative graded-simple associative unital $B$ with
graded involution $\iota$ and $\overline{(y_{ij})}=(y_{ij}^\iota)$
for $(y_{ij})\in \Mat_2(B)$;

\item[{\rm (IV)}] $\AC(q, S, F_0)$ for a graded-nondegenerate $q$
over a graded-field $F$ with Clifford-ample subspace $F_0$;
\end{enumerate}
\noindent or polarized
\begin{enumerate}
\item[{\rm (V)}] $\rmH_2(B,B_0, \pi)\oplus \rmH_2(B,B_0, \pi)$ for a
diagonal hermitian matrix system $\rmH_2(B,B_0,\pi)$ with
graded-simple noncommutative $B$;

\item[{\rm (VI)}] $\Mat_2(B)\oplus \Mat_2(B)$ for a noncommutative graded-simple
associative unital $B$ with $P(x)y=xyx$;

\item[{\rm (VII)}] $\AC(q, S, F_0)\oplus \AC(q, S, F_0)$ for
$\AC(q, S,F_0)$ as in {\rm (IV)}.
\end{enumerate}

\noindent Conversely, the Jordan triple systems in {\rm (I)--(VII)}
are graded-simple-trian\-gu\-lated.\end{corollary}

\begin{corollary}\label{divgradclassi} For a graded-triangulated Jordan triple system $J$
the following are equivalent: \sm

\begin{enumerate}
 \item[\rm (i)] $J$ is graded-simple and every homogeneous $0\neq m\in M$
     is invertible,

 \item[\rm (ii)] $J$ is division-triangulated,

 \item[\rm (iii)] $J$ is graded isomorphic to one of the following:
\begin{enumerate}
\item[{\rm (I)}] diagonal hermitian matrix system $\rmH_2(A,A_0, \pi , ^-)$
for a noncommutative division-graded $A$.
\item[{\rm (II)}] $\AC(q, S, F_0)$ for a graded-anisotropic $q$
over a graded-field $F$ with Clifford-ample subspace $F_0$.
\end{enumerate}
\end{enumerate}
\end{corollary}

\begin{proof} If (i) holds we can apply Th.~\ref{maintheorem}:
$J$ is graded isomorphic to a hermitian matrix system or to an ample
Clifford system, and $C$ is $u$-faithful. The assumption on $M$
together with Cor.~\ref{hermtor} and Cor.~\ref{clifforddivcor} then
show that $J$ is graded isomorphic to one of the two cases in (iii)
and that $J$ is division-triangulated. The implication (ii)
$\Rightarrow$ (i) is trivial, and (iii) $\Rightarrow$ (i) follows
from the quoted corollaries.
\end{proof}

\begin{corollary} \label{jtstorclassi} A graded Jordan triple system $J$ over a field $k$ is a
triangulated Jordan triple torus iff $J$ is graded isomorphic to
\begin{enumerate}
 \item[{\rm (I)}] a diagonal hermitian matrix system $\rmH_2(A, A_0, \pi, {}^-)$
for a noncommutative torus $A$, or to

 \item[{\rm (II)}] an ample Clifford system $\AC(D,q, M)$ with $D, M$
as described in {\rm Cor.~{\rm \ref{clifforddivcor}(b)}}.
\end{enumerate}
\end{corollary}

\begin{proof} This follows from Cor.~\ref{divgradclassi} and the description
of tori in Lemma~\ref{hermtor} and Cor.~\ref{clifforddivcor}.
\end{proof}

\begin{corollary} \label{jtszncla}
$J$ is a $\ZZ^n$-triangulated Jordan triple torus iff $J$ is graded
isomorphic to
\begin{enumerate}
 \item[{\rm (I)}] a diagonal hermitian matrix system $\rmH_2(A, A_0, \pi, {}^-)$ where
$A$ is a quantum $\ZZ^n$-torus, see {\rm Ex.~\ref{hermtorqun}} and
$\pi=\pi_{\rm rev}$ is the reversal involution,  or to

\item[{\rm (II)}] an ample Clifford system as described in
{\rm Ex.~\ref{examaczn}}. \end{enumerate}
\end{corollary}

\begin{proof} All statements follow from the quoted references. Note
that by construction $D_0$ generates $D$ in the Clifford case, so
that we are indeed in the setting of Ex.~\ref{examaczn}.
\end{proof}


\section{Graded-simple-triangulated Jordan algebras and Jordan pairs}
\label{sec:triangjp}

In this section we specialize our results on graded-triangulated
Jordan triple systems to Jordan algebras and Jordan pairs: We
classify graded-simple-triangulated Jordan algebras
(Th.~\ref{classalgth}) and Jordan pairs (Th.~\ref{classpairs}), and
we deduce from these theorems the classifications of
division-trian\-gu\-lated Jordan algebras and pairs
(Cor.~\ref{divalg}, Cor.~\ref{divpair}). As an example, we describe
the classification of $\ZZ^n$-triangulated Jordan algebra tori
(Cor.~\ref{jazncla}) and Jordan pair tori (Cor.~\ref{jordantori}).
\sm

In this paper all Jordan algebras are assumed to be unital, with
unit element denoted $1$ or $1_J$ if we need to be more precise, and
with Jordan product written as $U_xy$. A homomorphism of Jordan
algebras is a $k$-linear map $f : J \to J'$ satisfying $f(U_xy) =
U_{f(x)}f(y)$ and $f(1_J) = 1_{J'}$.

In order to apply our results we will view Jordan algebras as Jordan
triple systems with identity elements. Thus, to a Jordan algebra $J$
we associate the Jordan triple system $T(J)$ defined on the
$k$-module $J$ with Jordan triple product $P_xy=U_xy$. The element
$1_J\in J$ satisfies $P(1_J)={\rm Id}$. Conversely, every Jordan
triple system $T$ containing an element $1\in T$ with $P(1) = {\rm
Id}$ is a Jordan algebra with unit element $1$ and multiplication
$U_xy = P_xy$.

For many concepts there is no or not a big difference between $J$
and $T(J)$. For example, a Jordan algebra $J$ is  {\it graded by
$\La$\/} if and only if $T(J)$ is  graded by $\La$, in which case
$1_J \in J^0$. In this case, a graded ideal of $J$ is the same as a
graded   ideal of $T(J)$, and we will call $J$ {\it graded-simple}
if $T(J)$ is so. Moreover, if $e\in J$ is an idempotent, i.e.,
$e^2=e$, then $e$ is a tripotent of $T(J)$ and the Peirce spaces of
$J$ and $T(J)$ with respect to $e$ coincide, i.e.,
$T\big(J_i(e)\big)=T(J)_i(e)$, $i=0,1,2$. In particular, the Peirce
spaces $J_i(e)$ are graded  if $e\in J^0$. We thus get the following
corollary from Th.~\ref{grsiminh}.

\begin{corollary} \label{PeircsimpJA} If $J$ is a graded-simple Jordan algebra
with an idempotent  $0\ne e\in J^0$, then the Peirce space $J_2(e)$
is a graded-simple Jordan algebra, and if $J_0(e)\neq 0$ then
$J_0(e)$ is graded-simple too.\end{corollary}

A  graded   Jordan algebra $J$ is called {\it graded-triangulated}
by $(u;e_1,e_2)$ if $e_i=e_i^2\in J^0$, $i=1,2$, are supplementary
orthogonal idempotents and $u\in J_1(e_1)^0 \cap J_1(e_2)^0$ with
$u^2=1$ and $u^3=u$. It is called {\it $\La$-triangulated\/} if it
is graded-triangulated and $\supp_\La J$ generates $\La$ as a group.
Note that any $x\in J$ with $x^2=1$ satisfies $2x^3= x^2 \circ x =
2x$ and $U_{x^3-x} = U_x U_{x^2-1} = 0$ by \cite[(1.5.4) and
(3.3.4)]{ark}, whence $x^2=1$ implies $x^3=x$ in case ${1\over 2}\in
k$ or $x$ is homogeneous and $J$ is graded-nondegenerate. Of course,
we also have $x^2=1 \Rightarrow x^3=x$ if  $J$ is special. In any
case, with our definition of a triangle in a Jordan algebra, $J$ is
graded-triangulated by $(u;e_1,e_2)$ iff $T(J)$ is
graded-triangulated by $(u;e_1,e_2)$.  Also, $J$ is
$\La$-triangulated  iff $T(J)$ is so.

This close relation to graded-triangulated Jordan triple systems
also indicates how to get examples of graded-triangulated Jordan
algebras: We take a Jordan triple system which is
graded-triangulated by $(u; e_1, e_2)$ and require $P(e) = \Id$ for
$e=e_1 + e_2$. Doing this for our two basic examples, yields the
following examples of graded-triangulated Jordan algebras.

\begin{definition}\label{jatriadef} (A) {\it Hermitian matrix algebras}:
This is the  graded  Jordan triple system $\rmH_2(A, A_0,
\allowbreak \pi, ^-)$ of Def.~\ref{examplehermi} with automorphism
$^-=\Id$, which we will write as $\rmH_2(A,A_0,\pi)$. Note that this
is a Jordan algebra with product $U(x)y=P(x)y=xyx$ and identity
element $1_J=E_{11} + E_{22}$. If for example $A=B \boxplus B^{\rm
op}$ and $\pi$ is the exchange involution, then $\rmH_2(A, A_0, \pi)
\cong_\La \Mat_2(B)$ where $\Mat_2(B)$ is the Jordan algebra with
product $U_xy=xyx$.\sm

 (B) {\it Quadratic form Jordan algebras}: This is the  graded  ample
Clifford system $\AC(q,M,S,D, ^-, D_0)$ of Def.~\ref{cli} with
automorphism $^-=\Id$ and $S|_{M}=\Id$. Since then $P(e)=\Id$ we get
indeed a graded-triangulated Jordan algebra denoted $\AC\alg(q, M,
D, D_0)$ or just $\AC\alg(q,D, \allowbreak D_0)$ if $M$ is
unimportant. Note that this Jordan algebra is defined on $D_0e_1
\oplus M \oplus D_0e_2$ and has product $U_xy = q(x, \tilde y)x-
q(x)\tilde y$ where $q(d_1 e_1 \oplus m \oplus d_2 e_2) = d_1d_2 -
q(m)$ and $(d_1 e_1 \oplus m \oplus d_2e_2)\wt{} = d_2 e_1 \oplus -m
\oplus d_1 e_1$. (If ${1 \over 2} \in k$ it is therefore a reduced
spin factor in the sense of \cite[II, \S3.4]{taste}.)
\end{definition}

\begin{theorem} \label{classalgth} A graded-simple-triangulated Jordan algebra
 satisfying \begin{enumerate}
\item[{\rm (a)}] every nonzero $m\in M$ is a linear combination of invertible
homogeneous e\-le\-ments, or
\item[{\rm (b)}] the grading group $\La$ is torsion-free, \end{enumerate} is
graded isomorphic to one of the following Jordan algebras:
\begin{enumerate}
\item[{\rm (I)}] diagonal $\rmH_2(A,A_0, \pi )$ for a
graded-simple noncommutative $A$;
\item[{\rm (II)}] $\Mat_2(B)$ for a noncommutative graded-simple
associative  unital $B$;
\item[{\rm (III)}] $\AC\alg(q,F, F_0)$ for a graded-nondegenerate
$q $ over a graded-field $F$ with Clifford-ample subspace $F_0$.
\end{enumerate}

\noindent Conversely, all Jordan algebras in {\rm (I)--(III)} are
graded-simple-triangu\-la\-ted.
\end{theorem}

\begin{proof} Let $J$ be a graded-simple-triangulated Jordan algebra
 that satisfies (a) or (b). Then $T(J)$ is a graded-simple-triangulated Jordan
 triple system with $e=1_J$ satisfying (a) or (b) of Th.~\ref{maintheorem}, hence graded isomorphic as Jordan triple system
to $\rmH_2(C,C_0, \pi)$, where $(C,\pi)$ is graded-simple, or to
$\AC\alg(q, C, C_0)$, where $C$ is division-graded. But because the
graded isomorphisms appearing in Th.~\ref{maintheorem} preserve the
triangles, they are in fact isomorphisms of Jordan algebras.
Therefore $J$ is graded isomorphic to $\rmH_2(C,C_0, \pi)$, where
$(C,\pi)$ is graded-simple, or to $\AC\alg(q, C, C_0)$, where $C$ is
a graded-field. In the second case $J$ is of type (III) of the
statement for $F_0=C_0$ and $F=C$. On the other hand, it follows
from Prop.~\ref{relsimple} that $(C,\pi)$ is graded-simple iff
either $C$ is graded-simple or $C\cong_{\La} B\boxplus B$ for a
graded-simple associative $B$ and $\pi$ is the exchange involution.
Hence $\rmH_2(C,C_0, \pi)$, where $(C,\pi)$ is graded-simple, is as
in (I) or (II) of the statement. The converse follows from
Cor.~\ref{classification}.\end{proof}

\begin{definition}\label{jordalgtor} As in the Jordan triple system case, a
graded Jordan algebra $J$ is called {\it division-graded} if every
nonzero homogeneous element is invertible in $J$.

We say that $J$ is {\it division-triangulated} if it is
graded-triangulated, the Jordan algebras $J_i$, $i=1,2$, are
division-graded and every homogeneous $0\neq m\in M$ is invertible
in $J$. It is called {\it division-$\La$-triangulated\/} if it is
division-triangulated as well as $\La$-graded.

A division-($\La$)-triangulated Jordan algebra $J$ is called a {\it
($\La$)-triangulated Jordan algebra torus\/} if $J$ is defined over
a field $k$ and $\dim_k J_i^\la\le 1$ and $\dim_k M^\la\le 1$.

Thus, $J$ is a division-($\La$)-triangulated Jordan algebra iff
$T(J)$ is a division-($\La$)-triangulated Jordan triple system. We
therefore get the Jordan algebra versions of the
Corollaries~\ref{divgradclassi}--\ref{jtszncla}. We formulate the
first and last of them, and leave the translation of the second,
Cor.~\ref{jtstorclassi}, to the reader.
\end{definition}

\begin{corollary} \label{divalg} For a graded-triangulated Jordan
algebra $J$ the following are equivalent:

\begin{enumerate}
 \item[\rm (i)] $J$ is graded-simple and every homogeneous $0\neq m\in M$
     is invertible,

 \item[\rm (ii)] $J$ is division-triangulated,

 \item[\rm (iii)] $J$ is graded isomorphic to one of the following:
\begin{enumerate}
 \item[{\rm (I)}] a diagonal hermitian matrix algebra $\rmH_2(A,A_0, \pi)$
 for a noncommutative division-graded $A$;

 \item[{\rm (II)}] a quadratic form Jordan algebra $\AC\alg(q, F, F_0)$
 for a graded-anisotropic $q$ over a graded-field $F$ with
    Clifford-ample subspace $F_0$.
\end{enumerate}
\end{enumerate}
\end{corollary}

\begin{corollary} \label{jazncla} A graded Jordan algebra $J$ over a field $k$ is a
$\ZZ^n$-triangulated Jordan algebra torus iff $J$ is graded
isomorphic to
\begin{enumerate}
 \item[{\rm (I)}] a diagonal hermitian matrix algebra $\rmH_2(A, A_0, \pi)$
where $A$ is a noncommutative quantum $\ZZ^n$-torus and $\pi_{\rm
rev}$ is the reversal involution, see {\rm Ex.~\ref{hermtorqun}}, or
to

\item[{\rm (II)}] a quadratic form Jordan algebra
$\AC\alg(q,M,D,D_0) = D_0 e_1 \oplus M \oplus De_2$, where

\begin{enumerate} \item $D=k[\Ga]$ is the group algebra of a subgroup
$\Ga \subset \ZZ^n$ with $2\ZZ^n \subset \Ga$, hence $\Ga$ is free
of rank $n$ and $D$ is isomorphic to a Laurent polynomial ring in
$n$ variables,

\item $D_0$ is a Clifford-ample subspace, hence $D_0=D$ if ${\rm
char}(k) \ne 2$,

\item $M$ is a $\ZZ^n$-graded $D$-module which is free of finite rank,
with a homogeneous basis, say $\{u_0, \ldots, u_l\}$, and the $u_i$
have degree $\de_i \in \ZZ^n$ with $\de_0=0$ and $\de_i + \Ga \ne
\de_j + \Ga $ for $i\ne i$,

\item the $D$-quadratic form $q : M \to D$ satisfies $q(u_0)=1$,
$0\ne q(u_i)\in D^{2 \de_i}$ and $q(u_i, u_j) = 0$ for $i\ne j$.
\end{enumerate}
\end{enumerate}
\end{corollary}

\begin{remark} \label{jaznclarem} For $\La=\ZZ^n$ Cor.~\ref{jazncla} is proven in
\cite[Prop.~4.53 and Prop.~4.80]{AG} and in an equivalent form
(structurable algebras instead of Jordan algebras) in \cite[\S3,
Th.~9]{F}, assuming ${\rm char}\, k\ne 2,3$ (\cite{F}) or $k=\CC$ in
\cite{AG}.
\end{remark}

Let now $V=(V^+, V^-)$ be a Jordan pair. A  {\it grading of $V$ by
$\La$\/} is a decomposition $V^\sigma = \bigoplus_{\la \in \La}
V^\sigma[\la]$, $\sigma=\pm$, such that the associated polarized
Jordan triple system $T(V)$ is graded with homogeneous spaces
$T(V)^\la = V^+[\la] \oplus  V^-[\la]$. The properties of graded
Jordan triple systems we have considered so far in this paper make
sense for graded Jordan pairs too. For example, a graded Jordan pair
$V$ is {\it graded-nondegenerate} if every homogeneous absolute zero
divisor of $V^+$ or $V^-$ vanishes. It is immediate that $V$ is
graded-nondegenerate if and only if $T(V)$ is graded-nondegenerate.
As usual, $V$ is said to be {\it graded-simple} if
$Q(V^{\sigma})V^{-\sigma}\neq 0$ and every graded   ideal is either
trivial or equal to $V$; $V$ is {\it graded-prime} if it does not
contain nonzero graded ideals $I$ and $K$ such that
$Q(I^{\sigma})K^{-\sigma}= 0$ and {\it graded-semiprime} if
$Q(I^{\sigma})I^{-\sigma}= 0$ implies $I=0$. For properties defined
in terms of ideals we have the following lemma, whose proof is again
a straightforward adaptation of the proof in the ungraded situation.

\begin{lemma}\label{polarized} {\rm  (\cite[1.5]{NInv} for
$\La=0$)} Let $V$ be a   graded   Jordan pair.
\begin{enumerate}
\item[{\rm (i)}] If $I=(I^+,I^-)$ is a graded ideal of $V$, then $I^+\oplus
I^-$ is a graded ideal of\/ $T(V)$.
\item[{\rm (ii)}] Let $\pi^\sigma: T(V)\rightarrow V^\sigma$, $\si=\pm$,
be the canonical projection. If $J$ is a graded ideal of $T(V)$,
then
$$\underset{\widetilde {}}J:=(J\cap V^+,J\cap V^-)\quad \text{and}\quad
\tilde{J}:=\big(\pi^+(J),\pi^-(J)\big),$$ are graded  ideals of $V$
satisfying $\underset{\widetilde{}}J^+\oplus
\underset{\widetilde{}}J^-\subseteq J\subseteq \tilde{J}^+\oplus
\tilde{J}^-$ and
$$Q(\tilde{J}^\sigma)V^{-\sigma}+Q(V^{\sigma})\tilde{J}^{-\sigma}+
\{V^{\sigma}V^{-\sigma}\tilde{J}^\sigma\} \subseteq
\underset{\widetilde {}}J^{\sigma}$$
\item[{\rm (iii)}] $V$ is graded-(semi)prime or graded-simple if and
only if $T(V)$ is, respectively, graded-(semi)prime or
graded-simple.
\end{enumerate}\end{lemma}

Idempotents in a Jordan pair $V$ and tripotents in $T(V)$ correspond
to each other naturally: Any idempotent $e=(e^+,e^-)$ of $V$, i.e.,
$Q(e^{\sigma})e^{-\sigma}=e^{\sigma}$, gives rise to the tripotent
$e^+\oplus e^-$ of $T(V)$, and conversely any tripotent of $T(V)$
arises in this way. Moreover, we have the following obvious though
fundamental fact. If $V=V_2(e)\oplus V_1(e)\oplus V_0(e)$ is the
Peirce decomposition of $V$ with respect to an idempotent
$e=(e^+,e^-)$, then the Peirce spaces of $T(V)$ with respect to
$e^+\oplus e^-$ are $T(V)_i(e^+\oplus e^-)=T\big(V_i(e)\big),\quad
i=0,1,2$. The following corollary is a consequence of
Lem.~\ref{polarized}(iii) and Th.~\ref{grsiminh}.

\begin{corollary} \label{jppeircesim} If\/ $0\ne e\in V[0]$ is an idempotent of a
graded-simple Jordan pair $V$, then the Peirce space $V_2(e)$ is
graded-simple and if $V_0(e)\neq 0$, then $V_0(e)$ is also
graded-simple.
\end{corollary}

\begin{definition}\label{jptriagdef} Recall \cite{N3gr} that a triple of nonzero idempotents $(u; e_1,
e_2)$ of a Jordan pair $V$ is a {\it triangle} if $e_i \in
V_0(e_j)$, $i\neq j$, $e_i \in V_2(u)$, $i=1,2$, $u \in V_1(e_1)\cap
V_1(e_2)$, and the following multiplication rules hold for
$\sigma=\pm$: $Q(u\si)e\msi_i = e\si_j$, $i\ne j$, and $Q(e\si_1,
e\si_2)u\msi = u\si$.

A graded Jordan pair $V$ is said to be {\it graded-triangulated} if
$V$ contains a triangle $(u; e_1, e_2)$ in $V[0]$ and $V=V_1\oplus
M\oplus V_2 $, where $V_i=V_2(e_i)$, $i=1,2$, and $M=V_1(e_1)\cap
V_1(e_2)$. It is then immediate that $V$ is graded-triangulated if
and only if $T(V)$ is so. Naturally, we call $V$ {\it
$\La$-triangulated} if $T(V)$ is so. However, rather than applying
the Jordan triple classification to graded-($\La$)-triangulated
Jordan pairs, we will use the Jordan algebra Classification
Th.~\ref{classalgth}. Namely, it is well known \cite[p.470]{N3gr}
that  $V$ is covered by a triangle $(u; e_1, e_2)$ if and only if
$V\cong (J,J)$ where $J$ is the homotope algebra $J=V^{+(e^-)}$,
$e=e_1+e_2$,  with multiplication $U_xy=Q(x)Q(e^-)y$ and unit
element $e^+$, which is covered by the triangle $(u^+; e_1^+,
e_2^+)$. If $V$ is graded-($\La$)-triangulated then so is $J$, and
$V\cong_\La (J,J)$, since the isomorphism $V\cong (J,J)$ is given by
$(\Id, Q(e^-))$. Conversely, if $J$ is a graded-($\La$)-triangulated
Jordan algebra (or Jordan triple system), then the associated Jordan
pair $(J,J)$ is graded-($\La$)-triangulated. We therefore obtain the
following two types of examples of graded-triangulated Jordan pairs.
\end{definition}

\begin{example}[A$'$]\label{jptriadef} Hermitian matrix pairs:
{\rm $V=(J,J)$ where $J=\rmH_2(A, A_0, \pi)$ is a  hermitian matrix
algebra as in example (A) of Def.~\ref{jatriadef}.

(B$'$) {\it Quadratic form pairs}: $V=(J,J)$ where
$J=\AC\alg(q,D,D_0)$ is a quadratic form Jordan algebra as in
example (B) of Def.~\ref{jatriadef}. We note that $V\cong_\La
\big(\AC(q,{\rm Id}, D_0),\allowbreak \AC(q,{\rm Id},D_0)\big)$
where $\AC(q,{\rm Id},D_0)$ is the  corresponding ample Clifford
system. }\end{example}

\begin{theorem} \label{classpairs} A graded-simple-triangulated Jordan pair
 satisfying \begin{enumerate}
\item[{\rm (a)}] every nonzero $m\in M^\sigma$ is a linear combination of
invertible homogeneous elements, or
\item[{\rm (b)}] the grading group $\La$ is torsion-free, \end{enumerate} is
graded isomorphic to a Jordan pair $(J,J)$ where $J$ is a
graded-simple-triangulated Jordan algebra as described in {\rm
Th.~\ref{classalgth}}. Conversely, all these Jordan pairs $(J,J)$
are graded-simple-triangulated.
\end{theorem}

\begin{proof} Let $V$ be graded-simple-triangulated by $(u; e_1, e_2)$. Then $V$ is graded isomorphic to the Jordan pair of the unital Jordan algebra
$J=V^{+(e^-)}$. The algebra $J$ is then graded-simple with
$J^\la=V^+[\la]$, and graded-triangulated by $(u^+; e_1^+, e_2^+)$.
Thus $J$ is graded isomorphic to an algebra described in
Th.~\ref{classalgth}. Conversely, let $V=(J,J)$ be the Jordan pair
for $J$ as in (I)--(III). It follows from Th.~\ref{classification}
that the associated Jordan triple system $T(V)$ is
graded-simple-triangulated, hence $V$ is graded-simple-triangulated
by Lem. \ref{polarized}(iii).\end{proof}

\begin{definition} \label{jptordef}
A  graded   Jordan pair $V$ is said to be {\it division-graded} if
it is nonzero and e\-very nonzero e\-le\-ment in $V^\sigma[\lambda]$
is invertible in $V[\la]$. A graded-triangulated Jordan pair $V$
will be called {\it division-triangula\-ted}, respectively {\it
division-$\La$-triangulated\/}, if the associated Jordan triple
system $T(V)=V^+ \oplus V^-$ is so. Similarly, a
division-triangulated Jordan pair defined over a field $k$ is a {\it
triangulated Jordan pair torus\/} if $\dim_k V_i^\sigma[\lambda] \le
1$, $i=1,2$, and $\dim_k M^\sigma[\lambda] \le 1$ for $\sigma=\pm$.
The notion of a {\it $\La$-triangulated Jordan pair torus\/} is the
obvious one.

Since invertibility in $V$ is equivalent to inverti\-bi\-lity in the
unital Jordan algebra associated to $V$, we get the following
corollaries.
\end{definition}

\begin{corollary} \label{divpair} For a graded-triangulated Jordan
pair $V$ the following are equivalent:

\begin{enumerate}
 \item[\rm (i)] $V$ is graded-simple and every homogeneous
     $0\neq m\in M\si$, $\sigma = \pm$, is invertible,

 \item[\rm (ii)] $V$ is division-triangulated,

 \item[\rm (iii)] $V$ is graded isomorphic to the Jordan pair $(J,J)$ where
 $J$ is one of the Jordan algebras of {\rm Cor.~\ref{divalg}}.
\end{enumerate} \end{corollary}

\begin{corollary} \label{jordantori} $V$ is a triangulated Jordan
pair torus iff $V$ is graded isomorphic to a Jordan pair $(J,J)$
where $J$ is one of the following:
\begin{enumerate}
\item[{\rm (I)}] $J=\rmH_2(A,A_0, \pi)$ is a diagonal hermitian matrix algebra
of a noncom\-mu\-ta\-ti\-ve  torus $A$;

\item[{\rm (II)}] $J=\AC\alg(q, F, F_0)$ for a graded-anisotropic $q$
over a graded-field $F$ with Clifford-ample subspace $F_0$, with
$F=D$ and $M$ as described in {\rm Cor.~\ref{clifforddivcor}(b).}
\end{enumerate}
\end{corollary}

\section{Graded-simple Lie algebras of type $\rmb_2$}
\label{b2lie}

In this section we  apply our results on graded-simple-triangulated
 Jordan algebras and pairs from the previous
section~\ref{sec:triangjp} and obtain a classification of
$(\rmb_2,\La)$-graded-simple and centreless
division-$(\rmb_2,\La)$-graded Lie algebras in  Th.~\ref{secfinres}
and Th.~\ref{findivres}. In particular, we classify centreless Lie
tori of type $(\rmb_2,\La)$ in Cor.~\ref{finalres}.

We begin by recalling the relevant definitions from the theory of
root-graded Lie algebras (Def.~\ref{liedef} and
Def.~\ref{lietordef}). The link to triangulated Jordan structures is
given by the Tits-Kantor-Koecher construction, reviewed in
\ref{revtkk} in general and then in Prop.~\ref{liecla} for the
particular types of Lie algebras studied in this section.

Since we realize $\rmb_2$-graded Lie algebras as central extensions
of Tits-Kantor-Koecher algebras of triangulated Jordan pairs, {\it
we assume in this section that all Lie algebras, Jordan pairs and
related algebraic structures are defined over a ring $k$ in which
$2\cdot 1_k$ and $3\cdot 1_k$ are invertible.}

\begin{definition} \label{liedef}
Let $R$ be a finite reduced root system ($R$ could even only be
locally finite in the sense of \cite{LN}). We suppose that $0\in R$,
and denote by ${\mathcal{Q}(R)}$ the root lattice of $R$. A Lie
algebra $L$ over $k$ is called {\it $(R,\La)$-graded} if
\begin{enumerate}
\item[{\rm (1)}] $L$ has a compatible $\mathcal{Q}(R)$- and
$\La$-gradings, i.e., $L=\oplus_{\la \in \La}L^\la$ and $L =
\allowbreak \oplus_{\al \in \mathcal{Q}(R)}L_\al$ such that for
$L^\la_\al=L^\la\cap L_\al$ we have
$$L_\al=\oplus_{\la \in \La}L^\la_\al,\quad L^\la=\oplus_{\al \in
\mathcal{Q}(R)}L_\al^\la,\quad \text{and}\quad
[L_\al^\la,L_\be^\kappa]\subseteq L_{\al+\be}^{\la+\kappa},$$ for
$\la, \kappa \in \La$, $\al, \be \in \mathcal{Q}(R)$.

\item[\rm (2)] $\{\al\in \scQ(R) : L_\al \ne 0\}\subseteq R$.

\item[{\rm (3)}] for every $0\neq \al \in R$ the homogeneous space
$L_\al^0$ contains an element $e\ne 0$ that is {\it invertible} in
the following sense: There exists $f\in L_{-\al}^0$ such that $h=[ e
, f]$ acts on $L_\be$, $\be \in R$, by
\begin{equation} \label{eq:liedef1}
[h, x_\be]=\lan \be,\al\ch\ran x_\be,\quad x_\be\in
 L_\be \end{equation}
where $\lan \al,\be\ch\ran$ denotes the Cartan integer of the two
roots $\al,\be\in R$.

\item[{\rm (4)}] $L_0=\sum_{0\ne \al\in R}[L_\al, L_{-\al}]$, and
 $\{\la \in \La: L_\al^\la \ne 0 \text{ for some } \al \in R \}$
spans $\La$ as abelian group.
\end{enumerate}
\noindent It follows that $\{\al\in \scQ(R) : L_\al \ne 0\}= R$.
Also, any invertible element $e$ generates an ${\frak sl}_2$-triple
$(e, h, f)$.  If $\La$ is not spanned by the support, we will simply
speak of an {\it $R$-graded Lie algebra with a compatible
$\La$-grading\/}. An $(R,\La)$-graded or $R$-graded Lie algebra with
a compatible $\La$-grading  is {\it graded-simple} if it is
graded-simple with respect to the $\La$-grading. A Lie algebra  is
{\it $(R,\La)$-graded-simple}, or {\it $R$-graded-simple} if it is
$(R,\La)$-graded and graded-simple, respectively $R$-graded with a
compatible graded-simple $\La$-grading.\sm

The definition of a root-graded Lie algebra is taken from
\cite{N3gr}. Originally, root-graded Lie algebras were defined over
fields of characteristic $0$ by a different system of axioms,
\cite{BM} and \cite{BZ}. As explained in \cite[Remark 2.1.2]{N3gr},
an $R$-graded Lie algebra in the sense of  \cite{BZ} and \cite{BM}
is the same as an $R$-graded Lie algebra as defined above. A lot is
known about the structure of root-graded Lie algebras, see
\cite[5.10]{Npers} for a summary of results. We will use here that
$L$ is a Lie algebra graded by a $3$-graded root system $R$ iff $L$
is a central covering of the Tits-Kantor-Koecher algebra $\frK(V)$
of a Jordan pair $V$ covered by a grid whose associated root system
is $R$ (\cite[2.7]{N3gr}). \end{definition}

\begin{review} \label{revtkk} {\bf Review of TKK-algebras.}
{\rm  Recall, see e.g. \cite[1.5]{N3gr}, that the
Tits-Kantor-Koecher algebra\/ $\frK(V)$ of a Jordan pair $V$, in
short the {\it TKK-algebra of $V=(V^+, V^-)$\/}, is a $\ZZ$-graded
Lie algebra defined on the $k$-module
$$\frK(V) = V^- \oplus \de(V^+, V^-) \oplus V^+$$ where $\de(V^+,
V^-)$ is the span of all inner derivations $\de(x,y) = (D(x,y),
\allowbreak - D(y,x))$, $(x,y)\in V$,  of $V$. The $\ZZ$-grading
$\frK(V)= \bigoplus_{i\in \ZZ} \frK(V)_{(i)}$ is a {\it
$3$-grading\/} in the sense that it has support $\{0, \pm 1\}$,
namely $\TKK(V)_{(\pm 1)} = V^{\pm}$ and $\TKK(V)_{(0)} = \de(V^+,
V^-)$.  The Lie algebra product is determined by $[x^+,y^-] =
\de(x^+, y^-)$ and by the natural action of $\de(V^+, V^-)$ on
$V^\pm$: $[\de(x^+, y^-), u^+] = \{x^+,\, y^-,\, u^+\}$ and
$[\de(x^+, y^-), v^-] = - \{ y^-,\, x^+,\, v^-\}$. It is important
for the connection between Jordan theory and Lie algebras that,
conversely, for any $3$-graded Lie algebra $L=L_{(1)} \oplus L_{(0)}
\oplus L_{(-1)}$ the ``wings'' $(L_{(1)}, L_{(-1)})$ form a Jordan
pair $V_L$ with Jordan triple product
\begin{equation} \label{eq:3grpro}
   \{ x\si, \, y\msi, \, z\si\} =[[x\si, \,y\msi],\, z\si]
\end{equation}
for $x\si, z\si\in V\si=L_{(\sigma 1)}$ and $y\msi\in
V\msi=L_{(-\sigma 1)}$, $\sigma=\pm$. Moreover, the ideal $L'=
L_{(-1)} \oplus [L_{(-1)}, \, L_{(1)}] \oplus L_{(1)}$ of $L$ is a
central extension of the TKK-algebra $\TKK(V_L)$, namely $L'/C \cong
\TKK(V_L)$ for $$C=\{ d\in [L_{(-1)}, \, L_{(1)}] : [d, L_{\pm (1)}]
= 0 \}.$$ We note that because of $\frac{1}{2}, \frac{1}{3} \in k$,
a Jordan pair can be defined by the Jordan triple products $\{.,.,.
\}$. The formula (\ref{eq:3grpro}) is crucial: It allows one to
transfer properties between the Jordan pair and the associated Lie
algebras. An example is Prop.~\ref{liecla} below.

 A grading of $V$ by $\La$ extends to a grading of $\TKK(V)$ by $\La$
using the canonical grading of $\de(V^+, V^-)$. The gradings of
$\TKK(V)$ used in the following will all be induced in this way from
gradings of $V$. We point out that $\supp_\La V \subseteq \supp_\La
\TKK(V)$, but both span the same subgroup of $\La$. }\end{review}

\begin{definition} \label{lietordef} To define special cases of root-graded Lie algebras, we extend
the definition of an {\it invertible element\/} to any $e\in
L_\al^\la$, $\al \ne 0$, requiring the {\it inverse} $f\in
L_{-\al}^{-\la}$ and the equation (\ref{eq:liedef1}) for $h=[e,f]$.
 Then, an $(R,\La)$-graded or $R$-graded Lie algebra $L$ with a compatible $\La$-grading  is {\it division-graded}
if every nonzero element in $L_\al^\la$, $\al\ne 0$, is invertible,
and a {\it Lie torus} if it is division-graded, $k$ is a field and
$\dim_k L_\al^\la \le 1$ for all $0\ne \al\in R$. As usual, in this
case we will speak of a {\it division-$(R,\La)$-graded} Lie algebra
and a {\it Lie torus of type $(R,\La)$\/}, or {\it
division-$R$-graded} Lie algebras and {\it Lie tori of type $R$\/}
if $\La$ is not necessarily spanned by the $\La$-support.
\end{definition}
\sm

We now specialize $R=\rmb_2 = {\rm C}_2 = \{0\}\cup \{\pm \ep_1 \pm
\ep_2\} \cup \{ \pm 2\ep_1, \pm 2\ep_2\}$ and observe that $R$ is
$3$-graded with $1$-part $R_1=\{ 2\ep_1, \ep_1 + \ep_2, 2\ep_2\}$
(an isomorphic realization of this $3$-graded root system will be
used in Ex.~\ref{tkkqfexam}.

\begin{proposition} \label{liecla}
{\rm (a)} The TKK-algebra $\TKK(V)$ of a graded-triangulated Jordan
pair  $V=V_1 \oplus M \oplus V_2$  is $\rmb_2$-graded with
compatible $\La$-grading. Its homogeneous spaces are
\begin{align*}
  \TKK(V)_{\pm 2\ep_i}^\la &= V_i^\pm [\la], \\
  \TKK(V)_{\pm (\ep_1+\ep_2)}^\la &= M^\pm [\la],  \\
  \TKK(V)_{\ep_i-\ep_j}^\la  &= \delta (e_i^+,M^- [\la]),
           \quad (i,j)\in \{(1,2), (2,1)\}, \\
  \TKK(V)_0^\la &= \sum_{i=1,2}\delta (e_i^+ ,V_i^- [\la])+ \delta (u^+ ,M^-
     [\la]).
\end{align*}
Conversely, if $L$ is a $\rmb_2$-graded Lie algebra with compatible
$\La$-grading, then its centre $Z(L)$ is contained in $L_0$, namely
$Z(L) = \{ x\in L_0 : [x, L_\al]=0 \text{ for } \al \in (\pm
R_1)\}$, and $L/Z(L)$ is graded isomorphic to the TKK-algebra of the
graded-triangulated Jordan pair $V=(V^+, V^-)$ given by
$$
   V_i^{\pm}[\la] = L_{\pm 2\ep_i}^\la\quad\text{and}\quad
    M^\pm[\la] = L^\la_{\pm(\ep_1+\ep_2)}.
$$

{\rm (b)} Let $L$ be a $\rmb_2$-graded Lie algebra with compatible
$\La$-grading and let $V$ be the associated graded-triangulated
Jordan pair defined in {\rm (a)}. Then $L$ is
\begin{itemize}
\item[\rm (i)]   graded-simple if and only if
 $L=\TKK(V)$ and $V$ is graded-simple;

\item[\rm (ii)]  division-graded if and only if $V$ is
division-trian\-gu\-la\-ted;

\item[\rm (iii)] a Lie torus if and only if\/ $V$ is a triangulated Jordan pair
torus.

\item[\rm (iv)] $(\rmb_2,\La)$-graded iff\/ $V$ is
$\La$-triangulated. In particular, $L$ is
division-$(\rmb_2,\La)$-graded iff\/ $V$ is
division-$\La$-triangulated, and a Lie torus of type $(\rmb_2,\La)$
iff\/ $V$ is a $\La$-triangulated Jordan pair torus.
\end{itemize}
\end{proposition}

\begin{proof} (a) is the graded version of \cite[2.3--2.6]{N3gr}. The
generalization from $\La=\{0\}$ to arbitrary $\La$ is immediate. We
note that for a root-graded Lie algebra $L$ we have $L/L'$ and the
centre of $L$ coincides with the central subspace $C$ defined above.

In (b) the equivalence of graded-simplicity is a general fact
(\cite[2.5]{GN}), and (iii) is immediate from (ii) and the formulas
in (a). For the proof of (ii) one shows that $e\in L_\al$, $\al\in
R_{\pm 1}$, is invertible in the sense of Def.~\ref{liedef}, say
with inverse $f\in L_{-\al}$, iff $e$ has the appropriate
invertibility property in the Jordan pair $V$, again with inverse
$f$. This proves in particular the implication $\implies$. For the
other direction, it then suffices to show that for a root-graded Lie
algebra $L$ invertibility in the spaces $L_\al$, $\al \in R_1$,
forces invertibility in $L_\gamma$, $\gamma \in R_0$. This can be
done by using $L_\gamma = [L_\al, L_{-\be}]$ for appropriate
$\al,\be\in R_1$. We leave the details to the reader, in particular
in view of Rem.~\ref{Faulkner}.
\end{proof}

\begin{remark} \label{Faulkner} In characteristic $0$, a centreless Lie torus of
type $(\rmb_2, \ZZ^n)$ is the same as the centreless core of an
extended affine Lie algebra of type $\rmb_2$. The latter has been
studied in \cite[\S4]{AG} for $k=\CC$. Therefore, in this setting
the torus part of the theorem above is implicit in \cite[\S4]{AG}.

One can also define Lie algebras graded by non-reduced root systems.
A $\rmb_2$-graded Lie algebra is then a special type of a
$\rmbc_2$-graded Lie algebra. In the setting of $\rmbc_2$-graded Lie
algebras the torus version of  Prop.~\ref{liecla}  has been proven
in \cite[Th.~3]{F}, where Jordan pairs and Jordan algebras are
replaced by structurable algebras, and a triangulated Jordan algebra
torus by a so-called quasi-torus.
\end{remark}

\begin{remark}  We have formulated Prop.~\ref{liecla} in terms of
triangulated Jordan pairs, since it is in this setting that the
proposition can be generalized to describe division-$R$-graded Lie
algebras and Lie tori of type $R$ for any $3$-graded root system
$R$. We will however not need this here.  \end{remark}

\begin{remark} We have seen in Th.~\ref{classpairs} that a graded-triangulated
Jordan pair $V=(V^+, V^-)$ is isomorphic to the Jordan pair $(J,J)$
associated to a graded-triangulated Jordan algebra $J$. Since
isomorphic Jordan pairs lead to isomorphic TKK-algebras, one also
has the Jordan algebra version of Prop.~\ref{liecla}. We leave the
formulation to the reader. \end{remark}

\begin{remark} For easier comparison with the
literature (\cite{AG,BY,F}) we indicate how to ``find'' the
$\La$-triangulated Jordan algebra $J$ in a $\rmb_2$-graded Lie
algebra $L$, using the notation of above. As a $k$-module, $J=
\bigoplus_{\la \in \La} (L_{2\ep_1}^\la \oplus L_{\ep_1 + \ep_2}^\la
\oplus L_{2\ep_2}^\la)$. For $a,b\in J$, the Jordan algebra product
is given by $a \cdot b = \frac{1}{2}[[a, 1^-], b]$ where $1^- = f_1
+ f_2$ and $f_i\in L_{-2\ep_i}^0$ is the inverse of the invertible
element $e_i \in L_{2\ep_i}$ whose existence is guaranteed by
condition (3) in Def.~\ref{liedef}. The identity element of $J$ is
$1_J = e_1 + e_2$, the elements $e_i$ are idempotents of $J$, and
$J$ is triangulated by $(u; e_1, e_2)$ for $u$ the invertible
element in $L_{\ep_1 + \ep_2}^0$.
\end{remark}

Prop.~\ref{liecla} reduces the classification of the various types
of $\rmb_2$-graded Lie algebras to determining the TKK-algebras of
the corresponding triangulated Jordan pairs and Jordan algebras,
which we have described in the previous section
\S\ref{sec:triangjp}. Models for these TKK-algebras were given in
\cite{N3gr} for arbitrary triangulated Jordan pairs over arbitrary
rings. A more precise description can be obtained in the
graded-simple and division-graded case. To this end, we re-visit the
description of the TKK-algebras of the two types of Jordan pairs and
algebras that appear in the classification of division-triangulated
Jordan pairs and algebras in Cor.~\ref{divpair} and
Cor.~\ref{divalg}.

 \begin{example} \label{lieexamher}
The TKK-algebra of\/ the Jordan pair $(J,J)$ for $J=\rmH_2(A,\pi)$,
equivalently, of the Jordan algebra $J$. {\rm (\cite[4.2]{N3gr}).
Note that the ample subspace $A_0$ of the hermitian matrix algebra
is $A_0= \rmH(A,\pi)$ since $\frac{1}{2} \in k$. Let
$$
\fru_2(A,\pi):=\left\{ \left[ \begin{array}{cc}
a & b  \\
c & -a^{\pi t}
\end{array}\right] \in \Mat_4(A) :  \,a\in \Mat_2(A), b,c \in J\right\}.
$$
It is easy to see that $\fru_2(A,\pi)$ is the $-1$-eigenspace of an
involution of the associative algebra $\Mat_4(A)$, hence a
subalgebra of the general Lie algebra $\mathfrak{gl}_4(A)$. 
The natural $3$-grading of $\mathfrak{gl}_4(A)$ induces one of
$\fru_2 =\fru_2(A,\pi)$: We have $\fru_2=\fru_{2,\, (1)}\oplus
\fru_{2,\, (0)}\oplus \fru_{2,\, (-1)}$, where
\begin{eqnarray*}
\fru_{2,\, (1)} &=& \left\{ \left[ \begin{array}{cc} 0 & b \\ 0 & 0
\end{array}\right] : \,b \in \rmH_2(A,\pi)\right\}, \\
\fru_{2,\, (0)} &=& \left\{ \left[ \begin{array}{cc} a & 0 \\ 0 &
-a^{\pi t} \end{array}\right] : \,a \in \Mat_2(A)\right\}, \\
\fru_{2,\, (-1)}&=&\left\{ \left[ \begin{array}{cc} 0 & 0 \\ c & 0
\end{array}\right] : \,c \in \rmH_2(A,\pi)\right\}.
\end{eqnarray*}
The Lie algebra $\fru$ has a compatible $\La$-grading
$\fru=\bigoplus_{\la \in \La} \fru^\la$ for which $\fru^\la$
consists of the matrices with all entries in $A^\la$. The Jordan
pair associated to this $3$-graded Lie algebra, see the review
\ref{revtkk}, is $V=(J,J)$. We put
$$
  \frsu_2(A, \pi) = \fru_{2,\, (1)}\oplus [\fru_{2,\, (1)},\,
         \fru_{2, \, (-1)}] \oplus \fru_{2,\, (-1)},
$$
called the {\it symplectic Lie algebra associated to $(A,\pi)$\/}.
The proof of \cite[III, Prop.~4.2(a), (b)]{aabgp} also works in our
more general setting and yields
$$
   \frsu_2(A, \pi) =  [\fru_2(A, \pi),\, \fru_2(A, \pi)] =
     \{ X\in \fru_2(A, \pi) : \text{tr}(X) \in [A,A]\}.
$$
The Lie algebra $\frsu_2(A,\pi)$ is ${\rm C}_2$-graded with root
spaces indicated in the following tableau:
$$ \left[
\begin{array}{cc|cc}
        0 & \ep_1 -\ep_2 & 2\ep_1 & \ep_1 + \ep_2 \\
      \ep_2 - \ep_1  & 0 & \ep_1 + \ep_2 & 2\ep_2 \\ \hline
      -2\ep_1 & -\ep_1 - \ep_2  & 0 & \ep_2 - \ep_1 \\
       -\ep_1 - \ep_2 & -2\ep_2 & \ep_1-\ep_2 & 0
\end{array} \right]
$$
It follows from Prop.~\ref{liecla} that
$$
  \TKK(V) \cong \frsu_2(A,\pi) / Z\big( \frsu_2(A,\pi)\big).
  $$
However, one has the following criterion for $\frsu_2(A,\pi)$ to be
centreless (again \cite[III, Prop.~4.2(d)]{aabgp} works in our more
general setting): }\end{example}

\begin{lemma}\label{tkkc2} If $A= Z(A) \oplus [A,A]$, e.g. if
$A$  is a torus\/ {\rm (\cite[Prop.~3.3]{NY})}, then
$\frsu_2(A,\pi)$ is centreless and hence is (isomorphic to) the
TKK-algebra of the Jordan algebra $J=\rmH_2(A,\pi)$ and the Jordan
$(J,J)$. \end{lemma}

\begin{example} \label{tkkqfexam} The TKK-algebra of the Jordan
pair $V=( J,J)$ for $J=\AC\alg(q, D)$. {\rm (\cite[5.1, 5.3]{N3gr})
As in Ex.~\ref{lieexamher} the Clifford-ample subspace $D_0=D$ since
$\frac{1}{2}\in k$. Thus $J=De_1 \oplus M \oplus De_2$ for a
 graded  commutative associative $k$-algebra $D$ and $q : M \to
D$ is a $D$-quadratic form.

For a $D$-quadratic form $q_N : N \to D$ on a $D$-module $N$ we
define the {\it orthogonal Lie algebra of\/ $q_N$\/} as $
  \fro(q_N) = \{ X\in \End_D(N) : q_N(Xn,n)=0
  \,\text{ for all } n\in N\}$,
and the {\it elementary orthogonal Lie algebra $\eso(q_N)$\/} as $
\eso(q_N)= \text{Span}_D \{ n_1n_2^*-n_2n_1^* : n_1,n_2 \in N \}$
where $n_1^*$ is the $D$-linear form on $N$ defined by $n_1^*(n)=
q_N(n_1, n)$.

To describe the TKK-algebra of $V$ or, equivalently of $J$, we put
$h_1=e_1$, $h_{-1} = e_2$ and define a $D$-quadratic form $q_\infty$
on
$$ J_\infty = D h_2 \oplus Dh_1 \oplus M \oplus Dh_{-1} \oplus Dh_{-2}
$$
by requiring $q_\infty | M = -q$, $(D h_2 \oplus Dh_{-2}) \perp
(Dh_1 \oplus Dh_{-1}) \perp M$, and $q_\infty (h_i, h_{-i}) =1$,
$q_\infty (h_{\pm i})= 0$ for $i=1,2$. It follows from
\cite[(5.3.6)]{N3gr} that
$$  \TKK(V)\cong \eso(q_\infty),$$
in particular, $\TKK(V) \cong \fro(q_\infty)$ if $M$ is
 free of finite rank.
To obtain a more detailed description of the TKK-algebra, {\it we
assume in the following that $M$ has a homogeneous $D$-basis $\{ u_i
: i\in I\}$}, an assumption which by Cor.~\ref{clifforddivcor} is
always fulfilled if $J$ is division-triangulated. Then $J_\infty$ is
free too and endomorphisms of $J_\infty$ can be identified with
column-finite $(4+|I|) \times (4+|I|)$-matrices over $D$, which we
do with respect to the basis $h_2, h_1, (u_i)_{i\in I}, h_{-1},
h_{-2}$. Let $G$ be the $|I|\times |I|$-matrix representing $q$ with
respect to the basis $(u_i)_{i\in I}$. Then $X\in \eso(q_\infty)
\iff $
$$ X=
   \left[ \begin{array}{cc|c|cc}
      a & b & -m_2^t G & -s & 0 \\
      c & d & -m_1^t G & 0 & s \\ \hline
     \begin{array}{c} \phantom{00} \\ n_1 \\ \phantom{00}\end{array}  & n_2 &
     \begin{array}{ccc} \phantom{00} & X_M & \phantom{00}\end{array}
     & m_1 & m_2 \\ \hline
    t  & 0  & -n_1^t G  & d & b \\
    0  & -t &  -n_2^t G  &  c & a \end{array} \right]
$$
where $a,b,c,d,s,t\in D$, $m_1, m_2, n_1, n_2\in D^{(I)} \cong M$
and $X_M \in \eso(q)$ (if $M$ has finite rank the latter condition
is equivalent to $GX_M + X_M^tG=0$). The Lie algebra
$\eso(q_\infty)$ has a $\rmb_2$-grading for $\rmb_2=\{0\} \cup \{
\pm \ep_i, \pm \ep_2, \pm \ep_1 \pm \ep_2\}$ whose homogeneous
spaces $\fro(q_\infty)_\al$, $\al \in \rmb_2$, are symbolically
indicated by the matrix below.
$$
   \left[ \begin{array}{cc|c|cc}
      0 & \ep_2-\ep_1 & \ep_2    & \ep_1 + \ep_2 & \cdot  \\
     \ep_1 - \ep_2 & 0 & \ep_1  & \cdot & \ep_1 + \ep_2 \\ \hline
     \begin{array}{c} \phantom{00} \\ -\ep_2 \\ \phantom{00}\end{array}
         & -\ep_1 &
     \begin{array}{ccc} \phantom{00} & 0 & \phantom{00}\end{array}
     & \ep_1 & \ep_2 \\ \hline
    -(\ep_1 + \ep_2) & \cdot  & -\ep_1   & 0 & \ep_2-\ep_1 \\
    \cdot  & -(\ep_1 + \ep_2) &  -\ep_2  &  \ep_1-\ep_2& 0 \end{array} \right]
$$
Here $0$ is the $0$-root space, while $\cdot$ indicates an entry $0$
in the matrices in $\fro(q_\infty)$. The isomorphism $\TKK(V) \cong
\eso(q_\infty)$ is given by considering the $3$-grading of the root
system $\rmb_2$ whose $1$-part is $\{ \ep_2, \ep_2 \pm \ep_1\}$.
Hence $$V^{\pm} = \fro(q_\infty)_{\pm (\ep_1 + \ep_2)} \oplus
\fro(q_\infty)_{\pm \ep_2} \oplus \fro(q_\infty)_{\pm (\ep_2 -
\ep_1)}$$ are the right respectively left columns of the matrices in
$\eso(q_\infty)$. }\end{example} Before we state the main results of
this section, we remind the reader that we assume $\frac{1}{2},
\frac{1}{3}\in k$ in this section.

\begin{theorem} \label{secfinres} For a torsion-free $\La$ the
following are equivalent: \begin{enumerate}

\item[\rm (i)] $L$ is a $\rmb_2$-graded-simple Lie algebra,

\item[\rm (ii)] $L$ is graded isomorphic to the TKK-algebra of a
graded-simple-trian\-gu\-lated Jordan algebra,

\item[\rm (iii)] $L$ is graded isomorphic to one of the following Lie
algebras:

\begin{enumerate}
 \item[\rm (I)] $\frsu_2(A,\pi)/Z(\frsu_2(A,\pi))$ for a graded-simple
$A$ with involution $\pi$,

\item[\rm (II)] $\mathfrak{sl}_4(B)/Z\big( \mathfrak{sl}_4(B)\big)$
where $\mathfrak{sl}_4(B) = \{ X\in \mathfrak{gl}_4(B): {\rm tr}(X)
\in [B,B]\}$, and $B$ is a noncommutative graded-simple associative
algebra,

\item[\rm (III)] $\eso(q_\infty)$ in the notation of
{\rm Ex.~\ref{tkkqfexam}} for  $D=F$  a graded-field and $q: M \to
F$ a graded-nondegenerate $F$-quadratic form on a graded $F$-module
$M$ with base point $u\in M^0$.
\end{enumerate}
\end{enumerate}
\end{theorem}

\begin{proof} The equivalence of (i) and (ii) follows from Prop.~\ref{liecla}
and Th.~\ref{classpairs}. If (ii) holds, the cases (I) and (III) of
Th.~\ref{classalgth} correspond to the Lie algebras (I) and (III)
above, as follows from Ex.~\ref{lieexamher} and Ex.~\ref{tkkqfexam}.
That in case (II) of Th.\ref{classalgth} one gets case (II) above is
shown in \cite[(3.4.3)]{N3gr}. The remaining implication (iii)
$\Rightarrow$ (ii) is easy. \end{proof}

 With an analogous proof we obtain the classification of
$\rmb_2$-division-graded Lie algebras.

\begin{theorem} \label{findivres} For a Lie algebra $L$ the
following are equivalent: \begin{enumerate}

\item[\rm (i)] $L$ is a centreless division-$(\rmb_2,\La)$-graded Lie algebra,

\item[\rm (ii)] $L$ is graded isomorphic to the TKK-algebra of a
division-$\La$-triangu\-la\-ted Jordan algebra,

\item[\rm (iii)] $L$ is graded isomorphic to one of the following Lie
algebras:

\begin{enumerate}
 \item[\rm (I)] $\frsu_2(A,\pi)/Z(\frsu_2(A,\pi))$ where $A$ is a
noncommutative division-$\La$-graded associative algebra with
involution $\pi$ and generated by $\rmH(A,\pi)$,

\item[\rm (II)] $\eso(q_\infty)$ in the notation of
{\rm Ex.~\ref{tkkqfexam}} for $D=F$ a graded-field, $q: M \to F$ a
graded-anisotropic quadratic form on a graded F-module $M$ with base
point $u\in M^0$ and whose $\La$-support generates $\La$.
\end{enumerate}
\end{enumerate}
\end{theorem}

In particular, using Lem.~\ref{tkkc2} we get the following
corollary.

\begin{corollary} \label{finalres} A Lie algebra $L$ is a centreless Lie torus of type
$(\rmb_2,\La)$ iff $L$ is graded isomorphic to one of the following:
\begin{enumerate}
 \item[\rm (I)] A symplectic Lie algebra $\frsu_2(A,\pi)$ as in
{\rm Ex.~\ref{lieexamher}}, where $A$ is a noncommutative
 $\La$-torus with involution $\pi$ and generated by
$\rmH(A,\pi)$.

 \item[\rm (II)] An elementary orthogonal Lie algebra
$\eso(q_\infty)$ as in {\rm Ex.~\ref{tkkqfexam}} with $D$ a  torus,
 $M$ as described in {\rm Cor.~\ref{clifforddivcor}(b)},
$q: M \to D$ graded-anisotropic and $\La$ spanned by $\supp_\La M$.
\end{enumerate}
\end{corollary}

\begin{remark}  Centreless division-$(\rmb_2,\La)$-graded Lie algebras over
fields of characteristic $0$ and centreless Lie tori of type
$(\rmb_2,\La) $ are also described in \cite[Th.~4.3 and
Th.~5.9]{BY}, using a different method. Our approach gives a more
precise description of these Lie algebras. The special case
$\La=\ZZ^n$ had been established before in \cite[Th.~4.87]{AG}. It
could also be deduced from the results in \cite[Th.~9]{F}, see
Rem.~\ref{jaznclarem}.

The data $A$, $D$, $M$ and $q$ occurring for $\La=\ZZ^n$ in
Cor.~\ref{finalres} are described in detail in Cor.~\ref{jazncla}.
\end{remark}

\def\bysame{\leavevmode\hbox to3em{\hrulefill}\thinspace}
\bibliographystyle{amsalpha}

\end{document}